%% file: virloc.tex
\documentclass[a4paper,11pt,reqno]{amsart}

\input{preamble}
\input{macros}

\title{Virtual localization revisited\vspace{-2mm}}
\author[D. Aranha]{Dhyan Aranha}
\author[A.\,A. Khan]{Adeel A. Khan}
\author[A. Latyntsev]{Alexei Latyntsev}
\author[H. Park]{Hyeonjun Park}
\author[C. Ravi]{Charanya Ravi}
\date{2024-06-13}

\makeatletter
\def\l@subsection{\@tocline{2}{0pt}{4pc}{6pc}{}}
\def\l@subsubsection{\@tocline{3}{0pt}{8pc}{8pc}{}}
\makeatother

\begin{document}

\begin{abstract}
  Let $T$ be a split torus acting on an algebraic scheme $X$ with fixed locus $Z$.
  Edidin and Graham showed that on localized $T$-equivariant Chow groups, (a) push-forward $i_*$ along $i : Z \to X$ is an isomorphism, and (b) when $X$ is smooth the inverse $(i_*)^{-1}$ can be described via Gysin pullback $i^!$ and cap product with $e(N)^{-1}$, the inverse of the Euler class of the normal bundle $N$.
  In this paper we show that (b) still holds when $X$ is a quasi-smooth derived scheme (or Deligne--Mumford stack), using virtual versions of the operations $i^!$ and $(-)\cap e(N)^{-1}$.
  As a corollary we prove the virtual localization formula $[X]^\vir = i_* ([Z]^\vir \cap e(N^\vir)^{-1})$ of Graber--Pandharipande without global resolution hypotheses and over arbitrary base fields.
  We include an appendix on fixed loci of group actions on (derived) stacks which should be of independent interest.
  \vspace{-5mm}
\end{abstract}

\maketitle

\renewcommand\contentsname{\vspace{-1cm}}
\tableofcontents

\setlength{\parindent}{0em}
\parskip 0.75em

\thispagestyle{empty}

%%%%%%%%%%%%%%%%%%%%%%%%%%%%%%%%%%%%%%%%%%%%%%%%%%%%%%%%%%%%%%%%%%%%%%%%%%%

\section*{Introduction}

\subsection{Localization in equivariant intersection theory}

Let $X$ be a scheme\footnote{%
    In the introduction, all (derived) schemes are algebraic, i.e., of finite type over the base field $k$.
}.
Suppose $T$ is a split torus acting on $X$, and denote by $i : X^T \hook X$ the inclusion of the fixed point scheme \cite{Iversen}.
Let $\CH^T_*(X)$ denote the $T$-equivariant Chow groups\footnotemark~of Edidin--Graham \cite{EdidinGraham}, and $\CH^T_*(X)_\loc$ the localization at the set of first Chern classes of nontrivial $1$-dimensional representations of $T$.
Localization theory for the equivariant Chow groups, following Edidin--Graham \cite{EdidinGrahamLocalization}, consists of two salient statements.
\footnotetext{Throughout the paper, all Chow groups are taken with rational coefficients.}
The first asserts:

\begin{thm}[Concentration]\label{thm:eg1}
  Push-forward along $i : X^T \hook X$ induces an isomorphism
  \begin{equation*}
    i_* : \CH^T_*(X^T)_\loc \to \CH^T_*(X)_\loc.
  \end{equation*}
\end{thm}

The second statement concerns the description of the inverse of $i_*$ when $X$ is smooth.
In this case the fixed locus $X^T$ is also smooth\footnotemark, so $i : X^T \hook X$ is an lci closed immersion with normal bundle $N$.
\footnotetext{see \cite[Thm.~1.2]{Iversen} or \corref{cor:roasting}}
We thus have the operations of Gysin pull-back
\begin{equation}\label{eq:eggys}
  i^! : \CH^T_*(X) \to \CH^T_{*}(X^T)
\end{equation}
and cap product with the Euler class (= top Chern class) of $N$:
\begin{equation}\label{eq:ege}
  (-) \cap e(N) : \CH^T_*(X^T) \to \CH^T_{*}(X^T)
\end{equation}
Moreover, because $N$ has no fixed part (as it is the \emph{moving} part of the tangent bundle of $X$, restricted to $X^T$), one can show that \eqref{eq:ege} is invertible\footnotemark~on $\CH^T_*(X^T)_\loc$.
\footnotetext{see \cite[Prop.~4]{EdidinGrahamLocalization} or \propref{prop:vect} and \corref{cor:anisogamy}}
From the self-intersection formula
\begin{equation*}
  i^!i_*(-) = (-) \cap e(N)
\end{equation*}
we thus deduce:

\begin{cor}[Localization formula]\label{cor:eg2}
  If $X$ is smooth, we have the equality
  \begin{equation*}
    (i_*)^{-1} = i^!(-) \cap e(N)^{-1}
  \end{equation*}
  of maps $\CH^T_*(X)_\loc \to \CH^T_*(X^T)_\loc$.
  In particular, we have
  \begin{equation*}
    \alpha = i_*(i^!(\alpha) \cap e(N)^{-1}) \in \CH^T_*(X)_\loc
  \end{equation*}
  for every cycle class $\alpha \in \CH^T_*(X)$.
\end{cor}

\subsection{Virtual localization}

In this paper we will generalize the above two statements to the ``virtual'' context.
This means that $X$ will be cut out (locally) in an ambient smooth scheme $A$ by $n$ equations (or a section of a rank $n$ vector bundle $E$).
Thus it may be singular or of dimension greater than the virtual dimension
\begin{equation*}
  \mrm{vd} := \dim(A) - n = \rk(T_A) - \rk(E).
\end{equation*}
This is typical of ``obstructed'' moduli spaces, such as those arising in enumerative geometry, e.g. Kontsevich's moduli space of stable maps to a smooth projective variety $Y$ \cite{Kontsevich} which is used to define the Gromov--Witten invariants of $Y$.

Following M.~Kontsevich's insight in \cite{Kontsevich}, the notion of \emph{quasi-smoothness} in derived algebraic geometry provides a natural language in which to study virtual phenomena.
A quasi-smooth derived scheme $\mbf{X}$ may be thought of as a (possibly singular) classical scheme $X = \mbf{X}_\cl$ equipped with certain derived nilpotent data.
Intrinsically associated with $\mbf{X}$ is its \emph{virtual} tangent bundle, a perfect complex of amplitude $[0,1]$\footnote{%
  We use cohomological grading throughout the paper.
}.
In the local model, the term in degree $0$ is the tangent bundle of the ambient scheme $A$ and the term in degree $1$ is the obstruction bundle $E$.
In other words, $\bf{X}$ knows ``why'' $X$ is virtually smooth, and the virtual dimension is encoded as the Euler characteristic of its virtual tangent bundle.

We will prove the following virtual generalization of the equivariant localization theorem:

\begin{thmX}\label{thm:intro/virloc}
  Let $X$ be a quasi-smooth derived algebraic scheme with an action of a split torus $T$.
  Let $X^{hT}$ denote the homotopy fixed point scheme.
  Then push-forward along the canonical morphism $i : X^{hT} \to X$ induces an isomorphism
  \begin{equation*}
    i_* : \CH^T_*(X^{hT})_\loc \to \CH^T_*(X)_\loc
  \end{equation*}
  with inverse given by
  \begin{equation*}
    (i_*)^{-1} = i^!_T(-) \cap e_T(N^\vir)^{-1}.
  \end{equation*}
\end{thmX}

This will imply the following formula, computing the virtual fundamental class $[X]^\vir$ in terms of $[X^{hT}]^\vir$:

\begin{corX}\label{cor:intro/gp}
  Let $X$ be a quasi-smooth derived algebraic scheme with an action of a split torus $T$.
  Then we have
  \begin{equation*}
    [X]^\vir = i_* ([X^{hT}]^\vir \cap e_T(N^\vir)^{-1})
  \end{equation*}
  in $\CH^T_*(X)_\loc$.
\end{corX}

We now explain the various notation appearing in these statements.

\sssec*{Homotopy fixed points}

Recall that if $X$ is a $k$-scheme with $T$-action, the fixed point scheme $X^T$ can be described in terms of functors of points as follows.
For each $k$-algebra\footnotemark~$R$, we have an action of the group $T(R)$ on the set $X(R)$.
\footnotetext{All rings and algebras are implicitly commutative.}
An $R$-valued fixed point of $X$ is a point $x \in X(R)$ such that $t \cdot x = x$ in $X(R')$ for every $t \in T(R')$ and $R$-algebra $R'$.
The fixed point scheme $X^T$ represents the functor sending $R$ to the set of $R$-valued fixed points of $X$ (see \cite[\S 2]{Fogarty}).

If $X$ is a derived $k$-scheme, its functor of points $R \mapsto X(R)$ sends a \emph{derived} $k$-algebra\footnote{%
  The technical term is \emph{animated} commutative $k$-algebras.
  These are equivalently homotopy types of simplicial commutative $k$-algebras, or if $k$ is of characteristic zero, quasi-isomorphism types of connective commutative dg-$k$-algebras.
} $R$ to an \emph{\inftyGrpd} of $R$-points $X(R)$.
A $T$-action on $X$ amounts to $T(R)$-actions on $X(R)$ for each $R$, compatible up to coherent homotopy.
To define the correct replacement of $X^T$ in this setting, we take inspiration from homotopy theory (see \cite{ThomasonHomotopy}).
Roughly speaking, we say an $R$-valued \emph{homotopy fixed point} of $X$ is a point $x \in X(R)$, a ``fixing'' isomorphism $t \cdot x \simeq x$ in $X(R')$ for every $t \in T(R')$ and derived $R$-algebra $R'$, and a homotopy coherent system of compatibilities between the fixing isomorphisms as $t$ varies.
We will see that the functor of homotopy fixed points is represented by a derived scheme $X^{hT}$.
We refer to \ssecref{ssec:hofix} for the precise definition.

\sssec*{Push-forward}

When $R$ is an ordinary $k$-algebra, we have $X^{hT}(R) = X^T(R)$.
Thus the classical truncation of $X^{hT}$ is the fixed point scheme of the classical truncation $X_\cl$ (with respect to the induced $T$-action).
As Chow groups are insensitive to derived nilpotents, we have $\CH_*^T(X) \simeq \CH_*^T(X_\cl)$ and $\CH_*^T(X^{hT}) \simeq \CH_*^T((X_\cl)^T)$, and push-forward along the forgetful morphism $i : X^{hT} \to X$ is identified with push-forward along the inclusion $(X_\cl)^T \hook X_\cl$.
In particular, the invertibility of
\[ i_* : \CH^T_*(X^{hT})_\loc \to \CH^T_*(X)_\loc \]
follows from the classical concentration theorem for $X_\cl$ (\thmref{thm:eg1}).

\sssec*{Virtual fundamental class}

When $X$ is a smooth classical scheme, we have $X^{hT} = X^T$ (\propref{prop:ksno1o1}).
When $X$ is quasi-smooth, $X^{hT}$ is also quasi-smooth (\corref{cor:roasting}).
In particular, there are virtual fundamental classes $[X]^\vir$ and $[X^{hT}]^\vir$, respectively (see \cite{BehrendFantechi} or \cite[Constr.~3.5]{KhanVirtual}).

\sssec*{Virtual Gysin pull-back and Euler class}

The virtual fundamental classes $[X]^\vir$ and $[X^{hT}]^\vir$ can be regarded as the images of $[\pt] \in \CH^T_*(\pt)$ by virtual Gysin pull-back maps
\begin{equation*}
  p^! : \CH^T_*(\pt) \to \CH^T_*(X),
  \qquad q^! : \CH^T_*(\pt) \to \CH^T_*(X^{hT})
\end{equation*}
along the projections of $p : X \to \pt$ and $q : X^{hT} \to \pt$ (see \cite{Manolache} or \cite[Constr.~3.3]{KhanVirtual}).
If $i : X^{hT} \to X$ were quasi-smooth, we would similarly have a Gysin pull-back (compare \eqref{eq:eggys})
\begin{equation*}
  i^! : \CH^T_*(X) \to \CH^T_{*}(X^{hT})
\end{equation*}
such that $i^! \circ p^! = q^!$.
However, the essential difficulty in the virtual generalization of the localization theorem is that \emph{$i$ is rarely quasi-smooth when $X$ is not smooth}.

Similarly, we have no Euler class $e(N^\vir)$ as in \eqref{eq:ege} either, as the normal bundle $N^\vir$ (= the moving part of the virtual tangent bundle of $X$) is typically a \emph{perfect complex of amplitude $[0,1]$} rather than a vector bundle.

It turns out that $i^!$ and $e(N^\vir)$ do exist once we pass to localized Chow groups.
The following is the key technical result of the paper:

\begin{thmX}\label{thm:intro/gyseul}
  In the situation of \thmref{thm:intro/virloc}, we have:
  \begin{thmlist}
    \item\label{item:gyseul/gys}
    There exists a canonical Gysin pull-back map\footnote{%
      The subscript in $i^!_T$ is just a reminder that it only exists on the localizations.
      In case $i$ is actually quasi-smooth, then $i^!_T = i^!$.
    }
    \begin{equation*}
      i^!_T : \CH^T_*(X)_\loc \to \CH^T_{*}(X^{hT})_\loc
    \end{equation*}
    satisfying the functoriality formula
    \begin{equation*}
      i^!_T \circ p^! = q^! : \CH^T_*(\pt)_\loc \to \CH^T_{*}(X^{hT})_\loc.
    \end{equation*}
    
    \item\label{item:gyseul/eul}
    There exists a canonical Euler class operator
    \begin{equation*}
      (-) \cap e_T(N^\vir) : \CH^T_*(X^{hT})_\loc \to \CH^T_{*}(X^{hT})_\loc
    \end{equation*}
    which is invertible and satisfies the self-intersection formula
    \begin{equation*}
      i^!_T i_*(-) = (-) \cap e_T(N^\vir).
    \end{equation*}
  \end{thmlist}
\end{thmX}

Once \thmref{thm:intro/gyseul} is proven, Theorems~\ref{thm:intro/virloc} and \ref{cor:intro/gp} can be easily derived just as in the smooth case.
Note that $i^!_T$ and $e_T(N^\vir)$ are instances of more general constructions of virtual Gysin maps and Euler classes in $\CH^T_*(-)$; see Sects.~\ref{sec:gys} and \ref{sec:eul}, respectively.

\subsection{Relation with Graber--Pandharipande}
\label{ssec:gp}

In the form of \corref{cor:intro/gp}, the virtual localization theorem was first proven by Graber and Pandharipande \cite{GraberPandharipande}, when $k$ is the field of complex numbers and under the following technical assumptions:
\begin{enumerate}
  \item The scheme $X_\cl$ admits a $T$-equivariant closed immersion into an ambient smooth scheme,
  \item The virtual tangent bundle of $X$ admits a global resolution by $T$-equivariant vector bundles.
\end{enumerate}
The first assumption was later removed by H.-L.~Chang, Y.-H.~Kiem and J.~Li \cite{ChangKiemLi}.
Moreover, they relaxed the second assumption to the existence of a global resolution for the virtual normal bundle $N^\vir$ (= the moving part of the virtual tangent bundle).

When $N^{\vir}$ admits a global resolution $[N_0 \to N_1]$, we have $e_T(N^\vir) = e(N_0) \cup e(N_1)^{-1}$ (see \propref{prop:petal}).
This was taken as the \emph{definition} by \cite{GraberPandharipande} and \cite{ChangKiemLi}.
Although it was widely expected that \corref{cor:intro/gp} holds without the global resolution hypothesis on $N^\vir$, see e.g. \cite[Thm.~2.19, Rmk.~2.20]{JoyceEnumerative}, no general definition of $e_T(N^\vir)$ has appeared in the literature before as far as we know.
That is, \thmref{thm:intro/gyseul} is a prerequisite to even formulating the statements of \thmref{thm:intro/virloc} or \corref{cor:intro/gp}.
Even when a global resolution does exist, this is still a conceptual improvement on the literature in that the formulation is entirely intrinsic to the derived scheme $X$ (with the given $T$-action).

The global resolution hypothesis does hold for most concrete applications in enumerative geometry, though construction of explicit resolutions can often be annoying.
In some examples, such as the moduli space of semi-stable objects in a given abelian category, it is not clear whether a global resolution exists at all.
We refer to the recent work of Joyce \cite{JoyceEnumerative} for concrete examples.
In particular, Joyce's results make use of the generality of \corref{cor:intro/gp} (see Remark~2.20 of \emph{op. cit.}).

Recently, virtual fundamental classes have made appearances in the context of arithmetic geometry (see e.g. \cite{Madapusi,FengYunZhang}), as well as non-archimedean analytic geometry (see \cite{PortaYuGW}).
Whereas all previous work on virtual localization is restricted to the field of complex numbers, our approach is robust enough to apply over bases of arbitrary and even mixed characteristic, and apparently can also be adapted to derived analytic geometry as in \cite{PortaYuAnalytic}, where global resolutions are rare.

\subsection{Extension to stacks}

Our virtual localization theorem holds more generally for Deligne--Mumford stacks.\footnotemark
\footnotetext{In the introduction, all stacks are assumed of finite type over the base field $k$.}
Let us explain the subtleties that arise in this generality.

Consider first the non-virtual case, where $X$ is a \emph{smooth} Deligne--Mumford stack with $T$-action.
There is still a smooth closed substack $X^{hT} \sub X$ of homotopy fixed points.
However, the concentration theorem (\thmref{thm:eg1}) need not hold for $X^{hT}$.
For example, if $X = B\mu_n$ is the classifying stack of the group scheme of $n$th roots of unity, with the scaling action of $T=\bG_m$, then $X^{hT}$ is empty whereas $\CH^{\bG_m}_*(B\mu_n)_\loc$ is nontrivial.
In this case, we can solve the discrepancy by reparametrizing the action by the $n$-fold cover $T' = \bG_m \twoheadrightarrow \bG_m$, so that $X^{hT'} = X$.

When $T=\bG_m$ is rank one, Kresch showed that we can always find a reparametrization $T' = \bG_m \twoheadrightarrow \bG_m$ such that concentration holds for $X^{rhT} := X^{hT'}$.\footnotemark
\footnotetext{See \cite[Thm.~5.3.5]{Kresch}.
Note that Kresch assumed that the base field $k$ is algebraically closed.
Alper--Hall--Rydh showed how to construct the reparametrization over an arbitrary base field $k$, see \cite[Def.~5.25]{AlperHallRydhLuna}.
They also showed that $X^{rhT}$ is intrinisic to $X$ with its $T$-action, i.e., does not depend on the choice of reparametrization.
Finally, the first construction of $X^{rhT}$ when $X$ is \textit{singular} was given by Alper--Hall--Rydh; Kresch only constructed the reduced substack in general.}
Since $X$ and $X^{rhT}$ are smooth, we still have the self-intersection formula for $i : X^{rhT} \to X$ and one deduces a localization formula as in \corref{cor:eg2}, see \cite[Cor.~5.3.6]{Kresch}.

We will extend this to actions of higher rank split tori $T$, and to the virtual context.
We will show:

\begin{thmX}\label{thm:intro/virlocdm}
  Let $X$ be a quasi-smooth derived Deligne--Mumford stack with an action of a split torus $T$.
  Let $X^{rhT}$ denote the reparametrized homotopy fixed point stack (see \defnref{defn:XrhT}).
  Then push-forward along $i : X^{rhT} \to X$ induces an isomorphism
  \begin{equation*}
    i_* : \CH^T_*(X^{rhT})_\loc \to \CH^T_*(X)_\loc
  \end{equation*}
  with inverse given by
  \begin{equation*}
    (i_*)^{-1} = i^!_T(-) \cap e_T(N^\vir)^{-1}.
  \end{equation*}
  In particular, we have
  \begin{equation*}
    [X]^\vir = i_* ([X^{rhT}]^\vir \cap e_T(N^\vir)^{-1})
  \end{equation*}
  in $\CH^T_*(X)_\loc$.
\end{thmX}

The reparametrization $T' \twoheadrightarrow T$ is constructed such that $X^{rhT} = X^{hT'}$ contains all points of $X$ with $r$-dimensional $T$-stabilizer; see \ssecref{ssec:stab} for the notion of $T$-stabilizers.
The concentration statement then follows from the general concentration criterion of \cite[Thm.~B]{constack}.
The explicit formula for $(i_*)^{-1}$ is then derived as in the case of schemes, by proving a self-intersection formula for $i : X^{rhT} \to X$ and checking that $e_T(N^\vir)$ is still invertible in the Deligne--Mumford case.

For $T=\bG_m$, the above formula for $[X]^\vir$ was first proven by Graber--Pandharipande (see \cite[App.~C]{GraberPandharipande}) over the base field $k=\bC$, assuming the existence of a global smooth embedding and global equivariant resolution of the virtual tangent bundle as in \ssecref{ssec:gp}.
These assumptions were relaxed by Chang--Kiem--Li to the existence of a global equivariant resolution of the virtual normal bundle (see \cite[Thm.~3.5]{ChangKiemLi}).
In the smooth or non-virtual case, \thmref{thm:intro/virlocdm} extends \cite[Thm.~5.3.5, Cor.~5.3.6]{Kresch} to actions of higher rank tori and non-algebraically closed base fields.

\subsection{Outline}

In Sections~\ref{sec:eul} and \ref{sec:gys} we construct the Euler and Gysin operations of \thmref{thm:intro/gyseul}, respectively.
A key technical result is an upgraded homotopy invariance property (or Thom isomorphism) acquired by localized equivariant Chow theory (see \thmref{thm:homotopy}).
\thmref{thm:intro/virloc} and \corref{cor:intro/gp} are then derived in \secref{sec:virloc}.
Finally, \secref{sec:wall} contains an application of \corref{cor:intro/gp}: we prove a general form of the simple wall-crossing formula of \cite[\S 2.1, App.~A]{KiemLiWall}.

Appendix~\ref{sec:fixed} contains a detailed treatment of fixed loci of group actions on algebraic stacks, which is required for the proof of virtual localization for Deligne--Mumford stacks.
For completeness, we state many of the definitions and results for Artin stacks and for actions of general group schemes when possible.

In Appendix~\ref{sec:POT} we rewrite our main results in the language of perfect obstruction theories, as in \cite{BehrendFantechi} and \cite{GraberPandharipande}.
Recall that a quasi-smooth derived structure on a stack gives rise in particular to a perfect obstruction theory, a purely algebraic structure on the classical truncation.
While the language of quasi-smooth derived stacks is arguably more natural, in certain applications it may be convenient to apply the virtual localization theorem without needing to know that a given perfect obstruction theory actually comes from a derived structure.

\subsection{Conventions and notation}

We generally follow the conventions of the companion paper \cite{constack}.

\subsubsection{Stacks}

We fix a base commutative ring $k$, assumed noetherian and connected.
We write $\pt := \Spec(k)$.

We denote by $\Stk_k$ (resp. $\dStk_k$) the \inftyCat of (resp. derived) Artin stacks locally of finite type over $k$ with quasi-compact and separated diagonal.
Given a derived Artin stack $S \in \dStk_k$, we denote by $\Stk_S$ (resp. $\dStk_S$) for the \inftyCat of locally of finite type (resp. derived) Artin stacks over $S$ with quasi-compact and separated diagonal.
See \cite[\S 1]{constack} for our conventions on points and stabilizers of stacks.

We say that a derived $1$-Artin stack is \emph{quasi-Deligne--Mumford} if it has quasi-finite diagonal, or equivalently finite stabilizers.

\subsubsection{Chow groups and motivic Borel--Moore chains}

For every derived Artin stack $X \in \dStk_k$, we consider the complex
\[
  \CBM(X; \Q^\mot)\per
\]
of (periodized) motivic Borel--Moore chains, regarded as an object of the derived \inftyCat of $\bQ$-linear chain complexes.
By cap product, this is a module over the complex of motivic cochains $\Ccoh(X; \Q^\mot)\per$.
See \cite[\S 2]{constack} or \cite{KhanVirtual} for the definitions.

We define the (total) Chow group of $X$ by
\[
  \CH_*(X) := \H_0 \big(\CBM(X; \Q^\mot)\per\big).
\]
For global quotients $X = [U/G]$, $\CH_*(X) \simeq \CH_*(X_\cl)$ is identified with the equivariant Chow group $\CH_*^G(U_\cl) \otimes \Q$ of \cite{EdidinGraham}.
For $X$ Deligne--Mumford, or more generally $1$-Artin with finite inertia, $\CH_*(X)$ is identified with the Chow group of the coarse moduli space.

The complexes $\CBM(X)\per$ admit push-forwards along proper morphisms and (virtual) Gysin pull-backs along quasi-smooth morphisms, and these satisfy the base change formula for homotopy cartesian squares.
For any closed/open decomposition $i : Z \hook X$, $j : X\setminus Z \hook X$, there is a localization exact triangle
\begin{equation*}
  \CBM(Z)\per \xrightarrow{i_*}
  \CBM(X)\per \xrightarrow{j^!}
  \CBM(X\setminus Z)\per.
\end{equation*}
In particular, we have the derived invariance property
\[
  \CBM(X)\per \simeq \CBM({X_\cl})\per
\]
where $X_\cl$ is the classical truncation of $X \in \dStk_k$.

\subsubsection{Localized equivariant chains}

Throughout the paper, $T$ will be a split torus over $k$.
Given a \dA stack with $T$-action, we will denote the quotient stack in script font (e.g. $\sX = [X/T]$, $\sY = [Y/T]$, etc.).
If $T$ acts on $X \in \dStk_k$, we write
\begin{align*}
  \Chom^{\BM,T}(X; \Q^\mot)\per &:= \CBM(\sX; \Q^\mot)\per,\\
  \Ccoh_T(X; \Q^\mot)\per &:= \Ccoh(\sX; \Q^\mot)\per
\end{align*}
for the complexes of $T$-equivariant Borel--Moore chains and cochains, respectively.

We define the $T$-equivariant (total) Chow group of $X$ by
\begin{equation*}
  \CH^T_*(X) := \H_0 \big(\Chom^{\BM,T}(X; \Q^\mot)\per\big)
\end{equation*}
for any $X \in \dStk_k$.
When $X$ is an algebraic space, this is the $T$-equivariant Chow group of \cite{EdidinGraham}.
More generally, if $X$ is Deligne--Mumford or $1$-Artin with finite inertia, then $\CH^T_*(X)$ is identified with Kresch's Chow group of the quotient stack $\sX=[X/T]$ as in \cite{Kresch}.

We consider the localization
\[
  \Ccoh_T(X; \Q^\mot)_\loc := \Ccoh(\sX; \Q^\mot)_\loc := \Ccoh(\sX; \Q^\mot)\per[c_1(\Sigma)^{-1}]
\]
in the sense of higher algebra (see \cite[Prop.~7.2.3]{LurieHA}), where $\Sigma$ is the set of nontrivial $1$-dimensional $T$-representations (regarded as line bundles on $BT$), and similarly
\[
  \Chom^{\BM,T}(X; \Q^\mot)_\loc := \CBM(\sX; \Q^\mot)_\loc := \CBM(\sX; \Q^\mot)\per[c_1(\Sigma)^{-1}].
\]
We have
\begin{equation*}
  \CH^T_*(X)_\loc \simeq \H_0 \big(\Chom^{\BM,T}(X; \Q^\mot)_\loc\big),
\end{equation*}
where the left-hand side is the localization of $\CH^T_*(X)$ at the set of first Chern classes of nontrivial $1$-dimensional $T$-representations.
See \cite[\S 3]{constack} for details.

\subsubsection{Generalized chains}

In the text, we will fix more generally any oriented $\bQ$-linear commutative motivic ring spectrum $\Lambda$ over $k$.
The notation is generalized in the obvious way, e.g. we denote by
\[
  \CBM(X; \Lambda)\per
\]
the complex of (periodized) Borel--Moore chains with coefficients in $\Lambda$, for any $X \in \dStk_k$.
The case of motivic Borel--Moore chains considered above is obtained by taking $\Lambda = \bQ^\mot$ the motivic cohomology spectrum.
When there is no risk of ambiguity, we omit $\Lambda$ from the notation.

In this way, our results also apply to singular Borel--Moore homology or $\ell$-adic Borel--Moore homology, for example.
We can also take $\Lambda = \MGL_\Q$ the algebraic cobordism spectrum.
In that case we write
\[
  \Omega_*(X) := \H_0 \big(\CBM(X; \MGL_\Q)\per\big)
\]
for every $X \in \Stk_k$.
For global quotients $X = [U/G]$ over a field of characteristic zero, $\Omega_*(X)$ is identified with the equivariant algebraic bordism $\Omega_{*}^G(U_\cl) \otimes \Q$ of \cite{HellerMalagonLopez}.

\subsubsection{Relative chains}

It will be convenient in the paper to work with the complexes of
\emph{relative} Borel--Moore chains
\[
  \CBM(X_{/S}; \Lambda)\per
\]
for $X \in \dStk_S$ over a fixed $S \in \dStk_k$.
For $S=\pt := \Spec(k)$, we have $\CBM(X_{/\pt}; \Lambda)\per \simeq \CBM(X; \Lambda)\per$ by definition.
See \cite[\S 2]{constack} or \cite{KhanVirtual} for the definitions.

We also consider the following localized $T$-equivariant variants.
Given a $T$-equivariant morphism $X \to Y$ in $\dStk_k$, we write
\begin{equation*}
  \Chom^{\BM,T}(X_{/Y})_\loc := \Chom^{\BM}(\sX_{/\sY})_\loc := \Chom^{\BM}(\sX_{/\sY}) \otimes_{\Ccoh(BT)} \Ccoh(BT)_\loc.
\end{equation*}
Given $S \in \dStk_k$ and $X \in \dStk_S$ where $T$ acts on $X$ over $S$\footnotemark
\footnotetext{so that $X \to S$ is $T$-equivariant with respect to the trivial action on $S$}, we write
\begin{equation*}
  \CBM(\sX_{/S})_\loc := \CBM(\sX_{/S}) \otimes_{\Ccoh(\sX)} \Ccoh(\sX)_\loc.
\end{equation*}

\subsection{Acknowledgments}

We would like to thank Dominic Joyce and Marc Levine for their interest in this work, Andrew Kresch and Matthieu Romagny for discussions on fixed loci and corrections on a previous version, Harrison Chen for pointing out some relevant literature and comments on a draft, and Felix Janda for asking about \propref{prop:petal}.

We acknowledge support from the ERC grant QUADAG (D.A.), the NSTC grant 110-2115-M-001-016-MY3 (A.A.K.), the DFG through SFB 1085 Higher Invariants (C.R.), and the EPSRC grant no EP/R014604/1 (A.A.K. and C.R.).
We also thank the Isaac Newton Institute for Mathematical Sciences, Cambridge, for hospitality during the programme ``Algebraic K-theory, motivic cohomology and motivic homotopy theory'' where part of the writing was completed.
C.R. is grateful to Max Planck Institute for Mathematics in Bonn for its hospitality and financial support at the time of writing the paper.
This paper is part of a project that has received funding from the European Research Council (ERC) under the European Union's Horizon 2020 research and innovation programme (grant agreement No. 832833).

\section{Euler classes}
\label{sec:eul}

In this section we construct the Euler class operator of \thmref{thm:intro/gyseul}\itemref{item:gyseul/eul}.
See \constrref{constr:euler2} for the construction of the Euler class, \propref{prop:petal} for its compatibility with exact triangles, and \propref{prop:h0ubu01b} for its invertibility when the complex has no fixed part.
The self-intersection formula will be proven in the next section.

\subsection{Fixed and moving parts of complexes}
\label{ssec:rephosphorize}

Let $X$ be a derived stack over $k$.
Let $G$ be a diagonalizable group scheme of finite type over $k$.
Recall the following standard $\infty$-categorical version of \cite[Exp.~I, 4.7.3]{SGA3}:

\begin{prop}\label{prop:jbnso1}
  There is a canonical equivalence of stable \inftyCats
  \[
    \Dqc(X \times BG) \to \prod_\chi \Dqc(X),
  \]
  where the product is taken over characters $\chi : G \to \Gm$.
\end{prop}

More generally, let $G$ be an fppf group scheme acting trivially on $X$.
Then \v{C}ech descent along the cover $X \twoheadrightarrow [X/G] \simeq X \times BG$ yields a canonical equivalence between $\Dqc(X \times BG)$ and the \inftyCat of quasi-coherent $\cO_{G_X}$-comodules, where $G_X = G \times X$.
Now suppose that $G$ is \emph{diagonalizable}, so that $\cO_{G_X} \simeq \cO_X[M]$ is the group algebra of an abelian group $M$ (= the group of characters of $G$).
In this case one can argue as in the proof of \cite[Prop.~4.2]{Moulinos} to show that the \inftyCat of quasi-coherent $\cO_{G_X}$-comodules is equivalent to the \inftyCat
\[
  \Fun(M, \Dqc(X))
  \simeq \prod_{\chi\in M} \Dqc(X),
\]
where $M$ is regarded as a discrete set.

Given a quasi-coherent complex $\cF \in \Dqc(X \times BG)$, write $\cF^{(\chi)}$ for the $\chi$-eigenspace ($\chi\in M$), so that there are canonical isomorphisms
\begin{align*}
  \bigoplus_\chi \cF^{(\chi)} \to \cF,\\
  \cF \to \prod_\chi \cF^{(\chi)}.
\end{align*}
Indeed, the equivalence $F$ of \propref{prop:jbnso1} admits left and right adjoints $F^L$ and $F^R$, respectively, and these isomorphisms are the counit of $(F^L, F)$ and unit of $(F, F^R)$, respectively.

\begin{defn}\label{defn:uh1pbpbdf}
  The \emph{fixed part} of $\cF \in \Dqc(X \times BG)$ is its weight zero eigenspace and its \emph{moving part} is the direct sum of its nonzero weight eigenspaces.
  We write
  \[
    \cF^{\fix} := \cF^{(0)},
    \quad
    \cF^{\mov} := \bigoplus_{\chi\neq0} \cF^{(\chi)}.
  \]
\end{defn}

\subsection{Euler classes of vector bundles}

We fix a split torus $T$ over $k$.
The following can be regarded as a toy version of the localization theorem:

\begin{prop}\label{prop:vect}
  Let $X$ be a quasi-DM derived stack over $S \in \dStk_k$.
  Let $\cE$ be a locally free sheaf of rank $r$ on $\sX = X \times BT$ with no fixed part, i.e., $\cE^\fix \simeq 0$.
  Then we have:
  \begin{thmlist}
    \item
    The Euler class $e(\cE) \in \Ccoh(\sX)$ is invertible in $\Ccoh(\sX)_\loc$.

    \item
    Let $\pi : \sE = \V(\cE) \to \sX$ be the projection of the total space and $0 : \sX \to \sE$ the zero section.
    Then
    \begin{equation*}
      0_* : \CBM(\sX_{/S})_\loc \to \CBM(\sE_{/S})_\loc
    \end{equation*}
    is invertible, and we have canonical homotopies
    \begin{align*}
      (0_*)^{-1} \simeq 0^!(-) \cap e(\cE)^{-1} &~:~ \CBM(\sE_{/S})_\loc \to \CBM(\sX_{/S})_\loc,\\
      \pi^! \simeq 0_*(- \cap e(\cE)^{-1}) &~:~ \CBM(\sX_{/S})_\loc \to \CBM(\sE_{/S})_\loc.
    \end{align*}
  \end{thmlist}
\end{prop}

We begin with some standard identities.

\begin{lem}\label{lem:atmosphere}
  Let $X \in \dStk_k$ and $\pi : E = \V_X(\cE) \to X$ a vector bundle of rank $r$.
  Denote by $i : Z \to X$ the inclusion of the derived zero locus of any section $s : X \to E$.
  Then there is a canonical homotopy
  \begin{equation*}
    i_*i^!(-) \simeq (-) \cap e(\cE)
  \end{equation*}
  of maps $\CBM(X_{/X}) \to \CBM(X_{/X})\vb{r}$.
\end{lem}
\begin{proof}
  Consider the homotopy cartesian square
  \[\begin{tikzcd}
    Z \ar{r}{i}\ar{d}{i}
    & X \ar{d}{s}
    \\
    X \ar{r}{0}
    & E.
  \end{tikzcd}\]
  By the base change formula we have
  \begin{equation*}
    i_*i^! \simeq s^! 0_*.
  \end{equation*}
  Recall that $\pi^! : \CBM(X) \to \CBM(E)\vb{-r}$ is an isomorphism with inverse $0^!$, where $r=\rk(E)$.
  Since $s^!$ is also a left-sided inverse to $\pi^!$, we deduce that $s^! \simeq 0^!$, hence $i_*i^! \simeq 0^!0_* \simeq (-) \cap e(\cE)$.
\end{proof}

\begin{lem}\label{lem:enough}
  Let $X \in \dStk_k$.
  Let $\pi : E = \V_X(\cE) \to X$ be a vector bundle of rank $r$ with zero section $0 : X \to E$.
  Then we have the following canonical identities:
  \begin{align*}
    0_{*} 0^! &\simeq (-) \cap e(\pi^*(\cE)),\\
    0_{*} &\simeq \pi^!(- \cap e(\cE))
  \end{align*}
  of maps $\CBM(E_{/X}) \to \CBM(E_{/X})\vb{r}$ and $\CBM(X_{/X}) \to \CBM(E_{/X})$, respectively.
\end{lem}

\begin{proof}
  Let $\pi' : E \fibprod_X E \to E$ denote the second projection, regarded as a vector bundle over $E$ with zero section $(0,\id) : E \to E \fibprod_X E$.
  Note that $0 : X \to E$ may be regarded as the inclusion of the derived zero locus of the diagonal section $E \to E \fibprod_X E$.
  The first identity is then \lemref{lem:atmosphere}, and the second is derived by applying $\pi^!$ on the right.
\end{proof}

\begin{proof}[Proof of \propref{prop:vect}]
  Recall that $e(\cE) \in \Ccoh(X)$ is canonically identified with the image of the unit $1 \in \Ccoh(X)$ by the composite
  \begin{equation*}
    \CBM(X_{/X})
    \xrightarrow{0_*} \CBM(E_{/X})
    \xrightarrow{0^!} \CBM(X_{/X})\vb{r}
  \end{equation*}
  under the isomorphism $\Ccoh(X) \simeq \CBM(X_{/X})$, see \cite[Cor.~3.16]{KhanVirtual}.
  By homotopy invariance, $0^! \simeq (\pi^!)^{-1}$ is an isomorphism.
  Since $\cE^\fix \simeq 0$, the concentration theorem implies that $0_*$ is an isomorphism after localization (see \cite[Cor.~4.10]{constack}).
  It follows that $e(\cE)$ is invertible in $\Ccoh_T(X)_\loc$.

  Since $e(\cE)$ is invertible, \lemref{lem:enough} yields
  \[
    \pi^! \simeq 0_*(- \cap e(\cE)^{-1}).
  \]
  Since $\pi^!$ is also invertible with inverse $(\pi^!)^{-1} \simeq 0^!$, it also yields
  \[(0_*)^{-1} \simeq 0^!(-) \cap e(\cE)^{-1}\]
  as claimed.
\end{proof}

\subsection{Homotopy invariance}

To define Euler classes for $2$-term complexes, we will need the following generalized homotopy invariance property:

\begin{thm}\label{thm:homotopy}
  Let $X$ be a quasi-DM derived stack over $S \in \dStk_k$, regarded with trivial $T$-action.
  Let $\cE \in \Dperf^{T,\ge -1}(X)$ be a $T$-equivariant perfect complex whose fixed part $\cE^\fix$ belongs to $\Dperf^{T,\ge 0}(X)$.
  Then the projection $\pi : E := \V_{X}(\cE) \to X$ is quasi-smooth and the Gysin pull-back induces an isomorphism
  \[
    \pi^! : \Chom^{\BM}(\sX_{/S})_\loc \to \Chom^{\BM}(\sE_{/S})_\loc
  \]
  of $\Ccoh(\sX)_\loc$-modules.
\end{thm}

\begin{constr}\label{constr:0!}
  In the situation of \thmref{thm:homotopy}, we define the Gysin pull-back map
  \[ 0^!_T : \Chom^{\BM}(\sE_{/S})_\loc \to \Chom^{\BM}(\sX_{/S})_\loc \]
  as the inverse to $\pi^!$.
\end{constr}

\begin{proof}[Proof of \thmref{thm:homotopy}]
  Using the localization triangle and stratifying $X$ by global quotient stacks, we may assume that $X$ has the resolution property.

  Note that $\pi : E \to X$ factors through $\pi^\mov : E^\mov \to X$ and $E \to E^\mov$, which is a torsor under the vector bundle stack $\pi^\fix : E^\fix \to X$.
  By homotopy invariance for vector bundle stacks \cite[Prop.~2.20]{KhanVirtual}, we may therefore replace $\cE$ by $\cE^\mov$ and assume that $\cE$ has no fixed part.

  Since $X$ has the resolution property, we may argue as in the proof of \cite[Prop.~A.10]{KhanVirtual} by induction on the Tor-amplitude of the perfect complex $\cE$ to reduce to the case where
  \[ \cE = \Cofib(\cE^{-1} \to \cE^0) \in \Dperf^{T,[-1,0]}(X) \]
  with $\cE^{-1}, \cE^0 \in \Dperf^{T,[0,0]}(X)$.
  In this case we claim that
  \begin{equation}\label{eq:sap0d7gf0p1}
    \pi^!(-) = 0_{E,*}(-) \cap \pi^* \big( e(\cE^{-1}) \cap e(\cE^0)^{-1} \big).
  \end{equation}
  Recall that $e(\cE^0)$ and $e(\cE^{-1})$ are invertible by \propref{prop:vect} and $0_{E,*}$ is invertible by \cite[Cor.~4.10]{constack}, so this will imply that $\pi^!$ is invertible.

  Let us prove \eqref{eq:sap0d7gf0p1}.
  Note that the total space $E = \V(\cE)$ fits in a homotopy cartesian square
  \[ \begin{tikzcd}
    E \ar{r}{s}\ar{d}{\pi}
    & E_{0} \ar{d}{p}
    \\
    X \ar{r}{0_{E_1}}
    & E_1,
  \end{tikzcd} \]
  where $E_0 = \V(\cE^0)$ and $E_1 = \V(\cE^{-1})$.
  Recall the formulas
  \begin{align*}
    0_{E_0,*} &= \pi_{E_0}^!(-) \cap \pi_{E_0}^* e(\cE^0)\\
    0_{E_1,*} &= \pi_{E_1}^!(-) \cap \pi_{E_1}^* e(\cE^{-1})
  \end{align*}
  from \propref{prop:vect}.
  The second implies
  \[
    s_{*} \pi^! = \pi_{E_0}^!(-) \cap \pi_{E_0}^* e(\cE^{-1})
  \]
  by applying $p^!$ on the left and using the base change formula $p^! \circ 0_{E_1,*} \simeq s_* \circ \pi^!$.
  Since $s_*$ is an isomorphism by concentration (see \cite[Cor.~3.19, Cor.~4.9]{constack}), \eqref{eq:sap0d7gf0p1} now follows from the above identities.
\end{proof}

\subsection{Euler classes}

We define the Euler class of a perfect complex whose fixed part is a vector bundle.

\begin{constr}\label{constr:euler2}
  Let $X \in \dStk_k$ be a quasi-DM derived stack, regarded with trivial $T$-action.
  Let $\cE \in \Dperf^{T,\ge -1}(X)$ be a $T$-equivariant connective perfect complex whose fixed part $\cE^\fix$ belongs to $\Dperf^{T,\ge 0}(X)$.
  Let $E = \V_X(\cE)$ be the total space and let $0 : X \to E$ denote the zero section.
  The \emph{Euler class}
  $$e_T(\cE) \in \Ccoh_T(X)_\loc = \Ccoh(\sX)_\loc$$
  is the image by the $\Ccoh_T(\Spec(k))_\loc$-linear map
  \begin{equation}\label{eq:b1001b11sd}
    \CBM(\sX_{/\sX})_\loc
    \xrightarrow{0_*} \CBM(\sE_{/\sX})_\loc
    \xrightarrow{0^!_T} \CBM(\sX_{/\sX})_\loc
  \end{equation}
  of the unit $1 \in \Ccoh(\sX)_\loc$, under the isomorphism $\Ccoh(\sX)_\loc \simeq \CBM(\sX_{/\sX})_\loc$.
\end{constr}

\begin{prop}\label{prop:petal}
  Let $X \in \dStk_k$ be a quasi-DM derived stack, regarded with trivial $T$-action.
  Suppose given an exact triangle of $T$-equivariant connective perfect complexes in $\Dperf^{T,\ge -1}(X)$
  $$\cE' \to \cE \to \cE''$$
  whose fixed parts belong to $\Dperf^{T,\ge 0}(X)$.
  Then there is a canonical homotopy
  \[
    e_T(\cE) \simeq e_T(\cE') \cup e_T(\cE'')
  \]
  in $\Ccoh_T(X)_\loc$.
\end{prop}
\begin{proof}
  Write $E = \V_X(\cE)$ and denote by $\pi : E \to X$ and $0 : X \to E$ the projection and zero section, and similarly for $E'$ and $E''$.
  The given exact triangle gives rise to a cartesian square
  \[ \begin{tikzcd}
    E'' \ar{r}{i}\ar{d}{\pi''}
    & E \ar{d}{p}
    \\
    X \ar{r}{0'}
    & E'.
  \end{tikzcd} \]
  It will suffice to exhibit a canonical homotopy of maps
  \[
    0_T^! 0_*(-)
    \simeq 0_T'^! 0'_* \circ 0''^!_T 0''_*(-).
  \]
  By definition of $0_T^!$ it is enough to show that after applying $\pi^!$ on the left, the right-hand side is homotopic to $0_*$.
  This follows easily by combining the identity $\pi^! \simeq p^! \pi'^!$, the base change formula $p^! 0'_* \simeq i_* \pi''^!$, and the identity $i_* 0''_* \simeq 0_*$.
\end{proof}

\begin{prop}\label{prop:h0ubu01b}
  Let $X \in \dStk_k$ be a quasi-DM derived stack, regarded with trivial $T$-action.
  If $\cE \in \Dperf^{T,\ge -1}(X)$ has no fixed part, i.e., $\cE^\fix \simeq 0$, then $e_T(\cE) \in \Ccoh(\sX)_\loc$ is invertible.
\end{prop}
\begin{proof}
  This is equivalent to invertibility of the map \eqref{eq:b1001b11sd}, so we may use the localization triangle in Borel--Moore homology (and a stratification by quotient stacks) to reduce to the case where $X$ is classical and a quotient stack.
  In particular, $X$ admits the resolution property so we may represent $\cE$ as a bounded complex of finite rank locally free sheaves $\cE^i$ (each with no fixed part).
  In that case $e_T(\cE)$ is the cup product of $e(\cE^i)^{(-1)^i}$ by \propref{prop:petal}.
\end{proof}

\section{Gysin pull-backs}
\label{sec:gys}

In this section we construct the Gysin map of \thmref{thm:intro/gyseul}\itemref{item:gyseul/gys}.
See \constrref{constr:gys} for the construction, \propref{prop:funct} for the functoriality, and \propref{prop:self} for the self-intersection formula.
This will conclude the proof of \thmref{thm:intro/gyseul}.

\subsection{Specialization to the normal bundle}
\label{ssec:sp}

Following \cite[\S 3]{KhanVirtual}, we will construct virtual Gysin pull-backs using a construction called specialization to the normal bundle.
In \cite[Thm.~1.3]{KhanVirtual}, these were defined for quasi-smooth morphisms using a derived version of deformation to the normal cone.
In order to define Gysin pull-backs for non-quasi-smooth morphisms, we will require the more general deformation to the normal cone from \cite{HekkingKhanRydh}.

\begin{thm}\label{thm:D}
  Let $f : X \to Y$ be a homotopically finitely presented morphism in $\dStk_k$.
  Then there exists a commutative diagram of derived Artin stacks
  \begin{equation*}
    \begin{tikzcd}
      X \ar{r}{0}\ar{d}{0}
      & X \times \A^1 \ar[leftarrow]{r}\ar{d}{\widehat{f}}
      & X \times \bG_m \ar{d}{f\times \id}
      \\
      N_{X/Y} \ar{r}{i}\ar{d}
      & D_{X/Y} \ar[leftarrow]{r}{j}\ar{d}
      & Y \times \bG_m \ar[equals]{d}
      \\
      Y \ar{r}{0}
      & Y \times \A^1 \ar[leftarrow]{r}
      & Y \times \bG_m
    \end{tikzcd}
  \end{equation*}
  where each square is homotopy cartesian.
\end{thm}
\begin{proof}
  See \cite{HekkingKhanRydh}; we sketch the proof here.
  One defines $D_{X/Y}$ as the \emph{derived Weil restriction} of $X \to Y$ along the inclusion $0 : Y \hook Y \times \A^1$, or equivalently the derived mapping stack $\uMap_{Y\times\A^1}(Y \times \{0\}, X \times \A^1)$.
  It is easy to see that this derived stack fits in the homotopy cartesian diagram above.
  The nontrivial part is the algebraicity (i.e., that it is Artin).
  When $X$ and $Y$ are $1$-Artin and the base $k$ is of finite type over a field, then we can appeal to \cite[Thm.~5.1.1]{HalpernLeistnerPreygel}.\footnotemark
  \footnotetext{This is more than enough for the applications of virtual localization we have in mind.}
  In general, see \cite{HekkingKhanRydh}.
\end{proof}

\begin{constr}[Specialization]\label{constr:sp}
  Let $S\in\dStk_k$ and $f : X \to Y$ a homotopically of finite presentation morphism in $\dStk_S$.
  Associated with the closed/open decomposition
  \[
    N_{X/Y} \xrightarrow{i} D_{X/Y} \xleftarrow{j} Y \times \bG_m.
  \]
  we have the localization triangle
  \[
    \CBM({N_{X/Y}}_{/S}) \xrightarrow{i_*} \CBM({D_{X/Y}}_{/S}) \xrightarrow{j^!} \CBM(Y \times {\bG_m}_{/S}),
  \]
  whose boundary map
  \[
    \partial : \CBM(Y \times {\bG_m}_{/S})[-1] \to \CBM({N_{X/Y}}_{/S})
  \]
  gives rise to the \emph{specialization map}
  \begin{multline*}
    \sp_{X/Y} :
    \CBM(Y_{/S}) \xrightarrow{\mrm{incl}}
    \CBM(Y_{/S}) \oplus \CBM(Y_{/S})(1)[1]\\ \simeq
    \CBM(Y\times{\bG_m}_{/S})[-1] \xrightarrow{\partial}
    \CBM({N_{X/Y}}_{/S}),
  \end{multline*}
  where the splitting comes from the unit section $1 : Y \to Y\times\bG_{m}$.
\end{constr}

Recall that the localization triangle is compatible with proper push-forward and quasi-smooth Gysin maps:

\begin{lem}\label{lem:locprop}
  Let $S\in\dStk_k$ and suppose given a diagram
  \[\begin{tikzcd}
    Z' \ar{r}{i'}\ar{d}{f_Z}
    & X' \ar[leftarrow]{r}{j'}\ar{d}{f}
    & U' \ar{d}{f_U}
    \\
    Z \ar{r}{i}
    & X \ar[leftarrow]{r}{j}
    & U
  \end{tikzcd}\]
  of commutative squares in $\dStk_S$, where $(i,j)$ and $(i',j')$ are pairs of complementary closed and open immersions.

  \begin{thmlist}
    \item\label{item:afsdo0g09} If $f$, $f_Z$, and $f_U$ are proper, then there is a commutative diagram
    \[\begin{tikzcd}
      \CBM(Z'_{/S}) \ar{r}{i'_*}\ar{d}{f_{Z,*}}
      & \CBM(X'_{/S}) \ar{r}{j'^!}\ar{d}{f_*}
      & \CBM(U'_{/S}) \ar{d}{f_{U,*}}
      \\
      \CBM(Z_{/S}) \ar{r}{i_*}
      & \CBM(X_{/S}) \ar{r}{j^!}
      & \CBM(U_{/S}).
    \end{tikzcd}\]

    \item\label{item:afsdo0g10}
    If both squares are homotopy cartesian and $f$ is quasi-smooth, then there is a commutative diagram
    \[\begin{tikzcd}
      \CBM(Z_{/S}) \ar{r}{i_*}\ar{d}{f_Z^!}
      & \CBM(X_{/S}) \ar{r}{j^!}\ar{d}{f^!}
      & \CBM(U_{/S}) \ar{d}{f_U^!}
      \\
      \CBM(Z'_{/S}) \ar{r}{i'_*}
      & \CBM(X'_{/S}) \ar{r}{j'^!}
      & \CBM(U'_{/S}).
    \end{tikzcd}\]
  \end{thmlist}
\end{lem}

\begin{proof}
  The first statement is straightforward.
  For the second: the left-hand square commutes by base change for quasi-smooth Gysin maps, and the right-hand square commutes by functoriality of quasi-smooth Gysin maps.
\end{proof}

Using this we find that the specialization map commutes with proper push-forward and quasi-smooth pull-backs:

\begin{prop}\label{prop:sppush}
  Let $S\in\dStk_k$ and suppose given a commutative square $\Delta$
  \[\begin{tikzcd}
    X' \ar{r}{f'}\ar{d}{p}
    & Y' \ar{d}{q}
    \\
    X \ar{r}{f}
    & Y.
  \end{tikzcd}\]
  in $\dStk_S$.
  \begin{thmlist}
    \item\label{item:sppush}
    Suppose that the square is excessive, i.e., $q$ is proper, the square is cartesian on classical truncations, and the induced morphism $N_{X'/Y'} \to N_{X/Y} \fibprod_X X'$ is a closed immersion.
    Then there is a canonical homotopy
    \[
      N_{\Delta,*} \circ \sp_{X'/Y'} \simeq \sp_{X/Y} \circ q_*
    \]
    of maps $\CBM(Y'_{/S}) \to \CBM({N_{X/Y}}_{/S})$.

    \item\label{item:sppull}
    Suppose that $q$ and the induced morphism $N_\Delta : N_{X'/Y'} \to N_{X/Y}$ are both quasi-smooth.
    (For example, suppose $q$ is quasi-smooth and the square $\Delta$ is homotopy cartesian.)
    Then there is a canonical homotopy
    \[
      N_\Delta^! \circ \sp_{X/Y} \simeq \sp_{X'/Y'} \circ q^!
    \]
    of maps $\CBM(Y_{/S}) \to \CBM({N_{X'/Y'}}_{/S})\vb{-d}$, where $d$ is the relative virtual dimension of $q$.
  \end{thmlist}
\end{prop}
\begin{proof}
  Consider the commutative diagram
  \[\begin{tikzcd}
    N_{X'/Y'} \ar{r}\ar{d}{N_\Delta}
    & D_{X'/Y'} \ar{d}{D_\Delta}\ar[leftarrow]{r}
    & Y' \times \bG_m \ar{d}{q\times\id}
    \\
    N_{X/Y} \ar{r}\ar{d}
    & D_{X/Y} \ar[leftarrow]{r}\ar{d}
    & Y \times \bG_m\ar{d}
    \\
    \{0\} \ar{r}
    & \A^1\ar[leftarrow]{r}
    & \Gm.
  \end{tikzcd}\]
  For \itemref{item:sppush}, the assumption implies that $D_\Delta : D_{X'/Y'} \to D_{X/Y}$ is proper (see \cite{HekkingKhanRydh}).
  By construction of the specialization maps, it is enough to show the following square commutes:
  \begin{equation*}
    \begin{tikzcd}
      \CBM(Y'\times\Gm_{/S})[-1] \ar{rr}{\partial}\ar{d}{(q\times\id)_*}
      & & \CBM({N_{X'/Y'}}_{/S}) \ar{d}{N_{\Delta,*}}
      \\
      \CBM(Y\times\Gm_{/S})[-1] \ar{rr}{\partial}
      & & \CBM({N_{X/Y}}_{/S})
    \end{tikzcd}
  \end{equation*}
  where the horizontal arrows are the boundary maps in the respective localization triangles.
  Thus the claim follows from the compatibility of the localization triangle with proper direct image (\lemref{lem:locprop}\itemref{item:afsdo0g09}).

  For \itemref{item:sppull}, the assumptions imply that $D_\Delta$ is quasi-smooth.
  Indeed, this can be checked fibrewise, and both upper squares in the diagram are homotopy cartesian (since the lower squares and the left-hand and right-hand composite rectangles all are).
  It is enough to show the following square commutes:
  \[ \begin{tikzcd}
    \CBM(Y\times\Gm_{/S})[-1] \ar{rr}{\partial}\ar{d}{(q\times\id)^!}
    & & \CBM({N_{X/Y}}_{/S}) \ar{d}{N_\Delta^!}
    \\
    \CBM(Y'\times\Gm_{/S})[-1] \ar{rr}{\partial}
    & & \CBM({N_{X'/Y'}}_{/S}).
  \end{tikzcd} \]
  This follows from \lemref{lem:locprop}\itemref{item:afsdo0g10}.
\end{proof}

\begin{cor}\label{cor:spi_*}
  Let $S\in\dStk_k$ and $i : Z \to X$ a closed immersion in $\dStk_S$.
  Denote by $0 : Z \to N_{Z/X}$ the zero section of the derived normal bundle.
  Then there is a canonical homotopy
  \[ \sp_{Z/X} \circ i_* \simeq 0_* \]
  of maps $\CBM(Z_{/S}) \to \CBM({N_{Z/X}}_{/S})$.
\end{cor}
\begin{proof}
  Apply \propref{prop:sppush}\itemref{item:sppush} to the self-intersection square
  \[\begin{tikzcd}
    Z \ar[equals]{r}\ar[equals]{d}
    & Z \ar{d}{i}
    \\
    Z \ar{r}{i}
    & X
  \end{tikzcd}\]
  and note that $\sp_{Z/Z}=\id$.
\end{proof}

\subsection{Gysin pull-backs}

We now return to the $T$-equivariant situation.

\begin{defn}
  Let $f : X \to Y$ be a $T$-equivariant morphism in $\dStk_k$ where $T$ acts trivially on $X$.
  We say that $f$ is \emph{quasi-smooth in weight $0$} if the relative cotangent complex $L_{X/Y}$ lies in $\Dperf^{T,\ge -2}(X)$ and has fixed part $L_{X/Y}^\fix$ in $\Dperf^{T,\ge -1}(X)$.
\end{defn}

\begin{rem}
  Let $i : Z \to X$ be a $T$-equivariant closed immersion, where $T$ acts trivially on $Z$.
  If $i$ is quasi-smooth in weight 0, then the conormal complex $\cN_{Z/X} := L_{Z/X}[-1]$ lies in $\Dperf^{T,\ge -1}(Z)$ with fixed part in $\Dperf^{T,\ge 0}(Z)$.
  Thus we can form the Euler class $e_T(\cN_{Z/X}) \in \Ccoh_T(Z)_\loc$, and it is invertible when $Z$ is quasi-DM and $\cN_{Z/X}$ has no fixed part (\propref{prop:h0ubu01b}).
\end{rem}

\begin{exam}
  Let $X\in \dStk_k$ be quasi-compact Deligne--Mumford with $T$-action, and denote by $Z = X^{hT}$ the homotopy fixed point stack.
  Then the canonical morphism $\varepsilon : Z \to X$ is a closed immersion (\propref{Prop:Fixed_locus_DM_is_closed}) and $\cN_{Z/X}$ has no fixed part (\corref{cor:anisogamy}).
\end{exam}

\begin{constr}\label{constr:gys}
  Let $S\in \dStk_k$ and $f : X \to Y$ a $T$-equivariant morphism in $\dStk_S$, where $T$ acts trivially on $X$.
  Suppose that $X$ is quasi-DM and $f$ is quasi-smooth in weight 0.
  \begin{defnlist}
    \item
    The \emph{Gysin pull-back} map
    \begin{equation*}
      f^!_T : \Chom^{\BM}(\sY_{/S})_\loc \to \Chom^{\BM}(\sX_{/S})_\loc
    \end{equation*}
    is defined as follows.
    Consider the specialization map
    \[ \sp^T_{X/Y} : \CBM(\sY_{/S})_\loc \to \CBM([N_{X/Y}/T]_{/S})_\loc \]
    for $\sX \to \sY$ (\constrref{constr:sp}).
    Then $f^!_T$ is the composite
    \[
      \CBM(\sY_{/S})_\loc
      \xrightarrow{\sp_{\sX/\sY}} \CBM([N_{X/Y}/T]_{/S})_\loc
      \xrightarrow{0^!_T} \CBM(\sX_{/S})_\loc,
    \]
    where $0^!_T$ is as in \constrref{constr:0!}.

    \item
    The \emph{$T$-equivariant fundamental class} of $X \to Y$ is
    \begin{equation*}
      [X/Y]^T := [\sX/\sY] := f^!_T(1) \in \Chom^{\BM,T}(X_{/Y})_\loc
    \end{equation*}
    where $f^!_T$ is the Gysin pull-back with $S=[Y/T]$ and $1 \in \Ccoh_T(Y)_\loc$.
  \end{defnlist}
\end{constr}

\begin{rem}\label{rem:afh-p81}
  If $f$ is in fact \emph{quasi-smooth} (not just in weight 0), then $f^!_T$ is just the usual quasi-smooth Gysin pull-back $f^!$ of \cite[\S 3]{KhanVirtual} by construction.
\end{rem}

\subsection{Functoriality}

Fix $S\in \dStk_k$.

\begin{prop}\label{prop:f^!_T commutes with gys}
  Suppose given a $T$-equivariant homotopy cartesian square
  \[
    \begin{tikzcd}
    X' \ar{r}{f'}\ar{d}{p}
    & Y' \ar{d}{q}
    \\
    X \ar{r}{f}
    & Y
  \end{tikzcd} \]
  in $\dStk_S$, where $T$ acts trivially on $X$ and $X'$, $X$ and $X'$ are quasi-DM, $f$ is quasi-smooth in weight $0$, and $q$ is quasi-smooth.
  Then the square
  \[\begin{tikzcd}
    \CBM(\sY_{/S})_\loc \ar{r}{f^!_T}\ar{d}{q^!}
    & \CBM(\sX_{/S})_\loc \ar{d}{p^!}
    \\
    \CBM(\sY'_{/S})_\loc \ar{r}{f'^!_T}
    & \CBM(\sX'_{/S})_\loc
  \end{tikzcd}\]
  commutes, i.e., there is a canonical homotopy
  \[
    p^! \circ f^!_T \simeq f'^!_T \circ q^!
  \]
  of maps $\CBM(\sY_{/S})_\loc \to \CBM(\sX'_{/S})_\loc$.
\end{prop}
\begin{proof}
  Consider the commutative diagram
  \[ \begin{tikzcd}
    \CBM(\sY_{/S})_\loc \ar{r}{\sp_{\sX/\sY}}\ar{d}{q^!}
    & \CBM({N_{\sX/\sY}}_{/S})_\loc \ar[leftarrow]{r}{\pi^!}\ar{d}{N_p^!}
    & \CBM(\sX_{/S})_\loc \ar{d}{p^!}
    \\
    \CBM(\sY'_{/S})_\loc \ar{r}{\sp_{\sX'/\sY'}}
    & \CBM({N_{\sX'/\sY'}}_{/S})_\loc \ar[leftarrow]{r}{\pi'^!}
    & \CBM(\sX'_{/S})_\loc
  \end{tikzcd} \]
  where $N_p : N_{\sX/\sY} \to N_{\sX'/\sY'}$ is the induced morphism, and $\pi : N_{\sX/\sY} \to \sX$, $\pi' : N_{\sX'/\sY'} \to \sX'$ are the projections.
  The left-hand square commutes by \propref{prop:sppush}\itemref{item:sppull} and the right-hand square commutes by functoriality of quasi-smooth Gysin pull-backs \cite[Thm.~3.12]{KhanVirtual}.
  The claim thus follows by construction of the Gysin maps $f^!_T$ and $f'^!_T$.
\end{proof}

\begin{prop}\label{prop:funct}
  Let $f : X \to Y$ and $g : Y \to Z$ be $T$-equivariant morphisms in $\dStk_S$.
  Suppose that $X$ is quasi-DM and has trivial $T$-action, $f$ and $g\circ f$ are quasi-smooth in weight 0, and $g$ is quasi-smooth.
  Then we have:
  \begin{thmlist}
    \item
    There is a canonical identification
    \[
      [X/Z]^T \simeq [X/Y]^T \circ [Y/Z]^T \in \Chom^{\BM,T}(X_{/Z})_\loc.
    \]

    \item
    There is a commutative square
    \[ \begin{tikzcd}
      \CBM(\sZ_{/S})_\loc \ar{r}{g^!}\ar[equals]{d}
      & \CBM(\sY_{/S})_\loc \ar{d}{f^!_T}
      \\
      \CBM(\sZ_{/S})_\loc \ar{r}{(g\circ f)^!_T}
      & \CBM(\sX_{/S})_\loc
    \end{tikzcd} \]
    or in other words, a canonical homotopy
    \[
      (g \circ f)^!_T \simeq f^!_T \circ g^!
    \]
    of maps $\CBM(\sZ_{/S})_\loc \to \CBM(\sX_{/S})_\loc$.
  \end{thmlist}
\end{prop}
\begin{proof}
  The first claim follows from the second by taking $S=\sZ$ and evaluating on $1 \in \CBM(\sZ_{/\sZ})$.
  For the second, consider the following square:
  \[\begin{tikzcd}[matrix scale=0.8]
    \CBM(\sZ_{/S})_\loc \ar{rr}{(g\circ f)^!_T}\ar[equals]{d}
    &
    & \CBM(\sX_{/S})_\loc \ar[equals]{d}
    \\
    \CBM(\sZ_{/S})_\loc \ar{r}{g^!}\ar{d}{\sp_{Y/Z}}
    & \CBM(\sY_{/S})_\loc \ar{r}{f^!_T} \ar[equals]{d}
    & \CBM(\sX_{/S})_\loc \ar[equals]{d}
    \\
    \CBM({N_{\sY/\sZ}}_{/S})_\loc \ar{r}{0_{N_{\sY/\sZ}}^!}\ar[equals]{d}
    & \CBM(\sY_{/S})_\loc \ar{r}{f^!_T}
    & \CBM(\sX_{/S})_\loc \ar[equals]{d}
    \\
    \CBM({N_{\sY/\sZ}}_{/S})_\loc \ar{rr}{(0_{N_{Y/Z}}\circ f)^!_T}
    &
    & \CBM(\sX_{/S})_\loc
  \end{tikzcd}\]
  The middle left-hand square commutes by definition of the Gysin map $g^!$, and the middle right-hand square commutes tautologically.
  Therefore, to show that the upper rectangle commutes it is enough to show that the total outer composite square commutes, i.e.,
  \begin{equation}\label{eq:a0sf0p1b1}
    (0_{N_{Y/Z}}\circ f)^!_T \circ \sp^T_{Y/Z} \simeq (g\circ f)^!_T,
  \end{equation}
  and that the lower rectangle commutes, i.e.,
  \begin{equation}\label{eq:a0f8sh1}
    (0_{N_{Y/Z}}\circ f)^!_T \simeq f^!_T \circ 0_{N_{Y/Z}}^!.
  \end{equation}

  Let us show \eqref{eq:a0sf0p1b1}.
  Consider the following diagram of $T$-equivariant derived stacks over $S$:
  \[\begin{tikzcd}
    X \ar{r}{0}\ar{d}{f}
    & X \times \A^1 \ar[leftarrow]{r}{1}\ar{d}{f\times\id}
    & X \ar{d}{f}
    \\
    Y \ar{r}{0}\ar{d}{0_{N_{Y/Z}}}
    & Y \times \A^1 \ar[leftarrow]{r}{1}\ar{d}{\hat{g}}
    & Y \ar{d}{g}
    \\
    N_{Y/Z} \ar{r}\ar{d}{u}
    & D_{Y/Z} \ar{d}{v}\ar[leftarrow]{r}
    & Z \ar[equals]{d}
    \\
    Z \ar{r}{0}
    & Z \times \A^1 \ar[leftarrow]{r}{1}
    & Z
  \end{tikzcd}\]
  where each square is homotopy cartesian and $D_{Y/Z}$ is the derived deformation space (\thmref{thm:D}).
  Note that the morphism $h = \hat{g}\circ (f\times \id) : X\times\A^1 \to D_{Y/Z}$ is quasi-smooth in weight 0, since the conditions on Tor-amplitude can be checked on the derived fibres.
  Thus we have the following diagram
  \[\begin{tikzcd}
    \CBM(\sZ_{/S})_\loc\ar[leftarrow]{r}{0^!}\ar{d}{u^!}
    & \CBM(\sZ\times\A^1_{/S})_\loc\ar{r}{1^!}\ar{d}{v^!}
    & \CBM(\sZ_{/S})_\loc\ar[equals]{d}
    \\
    \CBM({N_{\sY/\sZ}}_{/S})_\loc\ar{d}{(0_{N_{Y/Z}}\circ f)^!_T}\ar[leftarrow]{r}
    & \CBM({D_{\sY/\sZ}}_{/S})_\loc\ar{d}{h^!_T}\ar{r}
    & \CBM(\sZ_{/S})_\loc\ar{d}{(g\circ f)^!_T}
    \\
    \CBM(\sX_{/S})_\loc\ar[leftarrow]{r}{0^!}
    & \CBM(\sX\times\A^1_{/S})_\loc\ar{r}{1^!}
    & \CBM(\sX_{/S})_\loc.
  \end{tikzcd}\]
  The two upper squares commute by functoriality of quasi-smooth Gysin maps, and the two lower squares commute by \propref{prop:f^!_T commutes with gys}.
  Moreover, by $\A^1$-homotopy invariance, the two upper and lower horizontal arrows $0^!$ and $1^!$ are isomorphisms and $0^! \simeq 1^!$.
  It follows that the left-hand and right-hand vertical composites are identified.
  Since $u^!\simeq \sp_{Y/Z}$ (by construction of quasi-smooth Gysin maps), this yields the desired homotopy \eqref{eq:a0sf0p1b1}.

  Let us show \eqref{eq:a0f8sh1}.
  By homotopy invariance for the vector bundle stack $\pi : N_{Y/Z} \to Y$, it is enough to show the claim after applying $\pi^!$ on the right, i.e.,
  \[ (0_{N_{Y/Z}}\circ f)^!_T \circ \pi^! \simeq f^!_T. \]
  By definition, $f^!_T$ and $(0_{N_{Y/Z}}\circ f)^!_T$ are the upper and lower composite arrows, respectively, in the following diagram:
  \[\begin{tikzcd}
    \CBM(\sY_{/S})_\loc\ar{r}{\sp_{\sY/\sZ}}\ar{d}{\pi^!}
    & \CBM({N_{\sY/\sZ}}_{/S})_\loc \ar{d}{q^!}\ar[leftarrow]{r}{p^!}
    & \CBM(\sX_{/S})_\loc \ar[equals]{d}
    \\
    \CBM({N_{\sY/\sZ}}_{/S})_\loc \ar{r}{\sp^T_{0_{N_{Y/Z}}\circ f}}
    & \CBM({N_{\sX/\sY}\oplus N_{\sY/\sZ}}_{/S})_\loc \ar[leftarrow]{r}{r^!}
    & \CBM(\sX_{/S})_\loc
  \end{tikzcd}\]
  where $p$, $q$ and $r$ are the projections (so that $p^!$, $q^!$, $r^!$ are invertible).
  The right-hand square commutes by functoriality of quasi-smooth Gysin maps for the composition
  \[
    r : N_{0_{N_{Y/Z}}\circ f} \simeq N_{X/Y} \oplus N_{Y/Z}
    \xrightarrow{q} N_{X/Y}
    \xrightarrow{p} X,
  \]
  and the left-hand square commutes by \propref{prop:sppush}\itemref{item:sppull} applied to the square
  \[\begin{tikzcd}
    X \ar{r}{0_{N_{Y/Z}}\circ f}\ar[equals]{d}
    & N_{Y/Z} \ar{d}{\pi}
    \\
    X \ar{r}{f}
    & Y,
  \end{tikzcd}\]
  where $\pi$ and $q: N_{X/Y} \oplus N_{Y/Z} \to N_{X/Y}$ are both smooth (the latter because $N_{Y/Z} \to Y$ is smooth, as $Y \to Z$ is quasi-smooth).
\end{proof}

\subsection{Self-intersection formula}

\begin{prop}\label{prop:self}
  Let $S\in\dStk_k$ and $i : Z \to X$ a $T$-equivariant closed immersion in $\dStk_S$.
  Suppose that $Z$ is quasi-DM with trivial $T$-action and $i$ is quasi-smooth in weight zero.
  Then we have an equality
  \begin{equation*}
    i^!_T \circ i_* = e(\cN_{\sZ/\sX}) \cap (-)
  \end{equation*}
  of maps $\CBM(Z_{/S})_\loc \to \CBM(Z_{/S})_\loc$.
\end{prop}
\begin{proof}
  We have
  \[ i^!_T \circ i_* = 0^!_T \circ \sp_{\sZ/\sX} \circ i_* \simeq 0^!_T \circ 0_* = e_T(\cN_{\sZ/\sX}) \cap (-), \]
  where the first equality holds by definition of $i^!_T$, the second holds by \corref{cor:spi_*}, and the last holds by definition of $e_T(\cN_{\sZ/\sX})$.
\end{proof}

\section{Virtual localization formula}
\label{sec:virloc}

In the previous two sections we proved our main technical result, \thmref{thm:intro/gyseul}.
We now apply this to deduce \thmref{thm:intro/virlocdm} (hence in particular \thmref{thm:intro/virloc} and \corref{cor:intro/gp}).
The statement for schemes or algebraic spaces is \thmref{thm:bordroom} and the generalization to DM stacks is \thmref{thm:qZBC}.
The proofs will be done in Subsects.~\ref{ssec:pbord} and \ref{ssec:pqZ}, respectively, as special cases of a more general statement (\thmref{thm:virlocgen}).

\subsection{Statements}
\label{ssec:virfix}

\begin{thm}\label{thm:bordroom}
  Let $X \in \dStk_k$ be a quasi-smooth derived algebraic space with $T$-action.
  Let $i : X^{hT} \to X$ denote the inclusion of the homotopy fixed point space (see \defnref{defn:yfgyoav1}).
  Then there is a canonical homotopy
  \begin{equation*}
    (i_*)^{-1} \simeq i_T^!(-) \cap e_T(\cN_{X^{hT}/X})^{-1}
  \end{equation*}
  of maps $\Chom^{\BM,T}(X)_{\loc} \to \Chom^{\BM,T}(Z)_{\loc}$, and a canonical homotopy
  \[[X] \simeq i_* ([X^{hT}] \cap e_T(\cN_{X^{hT}/X})^{-1} ) \]
  in $\Chom^{\BM,T}(X)_\loc$.
\end{thm}

\begin{thm}\label{thm:qZBC}
  Let $X\in\dStk_k$ be quasi-compact quasi-smooth Deligne--Mumford with $T$-action.
  Let $T' \twoheadrightarrow T$ be a reparametrization such that the canonical morphism $X^{hT'} \to X^T$ is surjective (\corref{cor:iouYVkdqDavArMOFS}) and let $i : X^{hT'} \to X$ be the canonical morphism (\defnref{defn:yfgyoav1}).
  Then there is a canonical homotopy
  \begin{equation*}
    (i_*)^{-1} \simeq i_T^!(-) \cap e_T(\cN_{X^{hT'}/X})^{-1}
  \end{equation*}
  of maps $\Chom^{\BM,T}(X)_{\loc} \to \Chom^{\BM,T}(Z)_{\loc}$, and a canonical homotopy
  \[[X] \simeq i_* ([X^{hT'}] \cap e_T(\cN_{X^{hT'}/X})^{-1} ) \]
  in $\Chom^{\BM,T}(X)_\loc$.
\end{thm}

\subsection{The general formula}

We will derive Theorems~\ref{thm:bordroom} and \ref{thm:qZBC} from the following more general statement, which gives a sufficient criterion for a given $T$-invariant closed substack to satisfy virtual localization.
Fix $S\in \dStk_k$.

\begin{thm}\label{thm:virlocgen}
  Let $X\in \dStk_S$ and $Z\sub X$ a $T$-invariant closed derived substack which is quasi-DM with trivial $T$-action.
  Suppose that the conormal sheaf $\cN_{Z/X}$ has no fixed part and for every geometric point of $X\setminus Z$ we have $\St^T_X(x) \subsetneq T_{k(x)}$.
  If $X$ is quasi-smooth over $S$, then
  \begin{equation*}
    i_* : \CBM(\sZ_{/S})_\loc \to \CBM(\sX_{/S})_\loc
  \end{equation*}
  is invertible and there is a canonical homotopy
  \begin{equation*}
    (i_*)^{-1} \simeq i_T^!(-) \cap e_T(\cN_{\sZ/\sX})^{-1}
  \end{equation*}
  of maps $\CBM(\sX_{/S})_{\loc} \to \CBM(\sZ_{/S})_{\loc}$, where $i : Z \to X$ is the inclusion.
\end{thm}
\begin{proof}
  First note that the assumptions imply that $Z$ is also quasi-smooth over $S$ (see \lemref{lem:barkpeel}).
  The concentration theorem in the form of \cite[Cor.~3.7]{constack} implies that the push-forward map
  \[ i_* : \CBM(\sZ_{/S})_\loc \to \CBM(\sX_{/S})_\loc \]
  is an isomorphism.
  By the self-intersection formula (\propref{prop:self}) we have
  \[
    i_T^!(i_*(-)) \cap e_T(\cN_{\sZ/\sX})^{-1}
    \simeq (-) \cap e_T(\cN_{\sZ/\sX}) \cap e_T(\cN_{\sZ/\sX})^{-1}
    \simeq (-).
  \]
  Thus $i_T^!(-) \cap e_T(\cN_{\sZ/\sX})^{-1}$ is a left-sided inverse to $i_*$, hence also a right-sided inverse.
\end{proof}

\begin{cor}\label{cor:-ph101213}
  Let the notation be as in \thmref{thm:virlocgen}.
  Then we have a canonical homotopy
  \begin{equation}\label{eq:jnaob}
    [\sX_{/S}] \simeq i_*([\sZ_{/S}] \cap e_T(\cN_{\sZ/\sX})^{-1})
  \end{equation}
  in $\CBM(\sX_{/S})_{\loc}$.
\end{cor}
\begin{proof}
  Let $p : \sX \to S$ and $q : \sZ \to S$ denote the projections.
  By \propref{prop:funct}, we have a canonical homotopy
  \[i^!_T [\sX_{/S}] = i^!_T \circ p^!(1) \simeq q^!_T(1) = [\sZ_{/S}],\]
  where $q^!_T = q^!$ since $q$ is quasi-smooth (see \remref{rem:afh-p81}).
  Thus the right hand side of \eqref{eq:jnaob} is
  \begin{align*}
    i_*([\sZ_{/S}] \cap e_T(\cN_{\sZ/\sX})^{-1})
    &= i_*(i^!_T([\sX_{/S}]) \cap e_T(\cN_{\sZ/\sX})^{-1})\\
    &= [\sX_{/S}]
  \end{align*}
  where the second equality is \thmref{thm:virlocgen}.
\end{proof}

\subsection{Proof of \thmref{thm:bordroom}}
\label{ssec:pbord}

Recall that $i : X^{hT} \to X$ is $T$-equivariant with respect to the trivial action on $X^{hT}$ (\remref{rem:trivact}).
By \corref{cor:anisogamy}, the conormal sheaf $\cN_{X^{hT}/X}$ is the moving part of $L_X|_{X^{hT}}$.
Since the classical truncation of $X^{hT}$ is identified with the fixed locus $X^{T}$ (\propref{prop:ksno1o1}), $X^{hT}$ contains all geometric points $x$ of $X$ with $\St^T_X(x) = T_{k(x)}$.
Thus we may apply \thmref{thm:virlocgen} and \corref{cor:-ph101213}.

\subsection{Proof of \thmref{thm:qZBC}}
\label{ssec:pqZ}

For the Deligne--Mumford case, we will require the following statement about invariance of Borel--Moore homology under $BG$-torsors.

\begin{thm}\label{thm:torsor}
  Let $S\in\dStk_{k}$, $X,Y \in \dStk_S$, and $G$ a finite group scheme\footnote{%
    i.e., a group stack which is finite schematic over $X$
  } of multiplicative type over $X$.
  Then for every $BG$-torsor ($G$-gerbe) $f : Y \to X$ over $S$, the direct image map
  \[ f_* : \CBM(Y_{/S}) \to \CBM(X_{/S}) \]
  is invertible.
\end{thm}

\begin{cor}
  Let $S\in\dStk_{k}$, $X \in \dStk_S$, and $G$ a finite group scheme of multiplicative type over $k$.
  Then direct image along the proper morphism $X \times BG \to X$ induces an isomorphism
  \[ \CBM(X\times BG_{/S}) \to \CBM(X_{/S}). \]
\end{cor}

\begin{proof}[Proof of \thmref{thm:torsor}]
  By étale descent for $\D$ we may assume that $X$ is affine, $G$ is diagonalizable, and $f$ is trivial, i.e. $Y = X \times BG$.
  Note that if $G$ is a product $H \times H'$ (of group schemes over $X$), then by base change it is enough to show the claim for $G=H$ and $G=H'$.
  Since $G$ is diagonalizable it is enough to treat $G=\mu_{n,X}$, hence either $G$ is finite étale over $X$ or $G_\red \simeq X$.
  In the latter case the claim follows from nilpotent invariance of $\CBM$.

  If $G$ is finite étale over $X$, we may use étale descent again to reduce to the case where $G$ is finite discrete.
  Thus let $G$ be a finite discrete group scheme and let us show that $f_* : \CBM(BG_{/S}) \to \CBM(X_{/S})$ is invertible.
  Let $s : X \twoheadrightarrow BG$ denote the quotient map.
  Consider the simplicial diagram of direct image maps
  \[ \CBM(G^{\bullet+1}_{/S})_{hG} \to \CBM(G^\bullet_{/S}) \]
  where $G^\bullet$, resp. $G^{\bullet+1}$ is the \v{C}ech nerve of $X \twoheadrightarrow BG$, resp. $G \twoheadrightarrow X$.
  This is an isomorphism in every degree by finite Galois codescent (since $G$ acts freely on $G$ and its iterated fibre powers over $S$).
  By proper codescent\footnotemark
  \footnotetext{
    This is the assertion that for $S \in \dStk_k$ fixed, $X \mapsto \CBM(X_{/S})$ satisfies codescent along proper surjections.
    It can be proven as in \cite[Thm.~5.7]{KhanKstack}, using finite étale descent (see \cite[Thm.~14.3.4]{CisinskiDegliseBook}) and topological invariance (see \cite[Cor.~2.1.9]{ElmantoKhan}) as ingredients.
  }
  , passing to the colimit gives rise to the isomorphism
  $$s_* : \CBM(X_{/S}) \simeq \CBM(X_{/S})_{hG} \to \CBM(BG_{/S})$$
  where the isomorphism on the left is because $G$ acts trivially on $X$.
  Since $f_* s_* \simeq \id$, it follows that $f_*$ is also invertible as claimed.
\end{proof}

\begin{proof}[Proof of \thmref{thm:qZBC}]
  For any $Y \in \dStk_k$ with $T$-action, it follows from \thmref{thm:torsor} that the morphism $f : [Y/T'] \to [Y/T]$
  induces an isomorphism
  \[
    f_* : \Chom^{\BM,T'}(Y)_\loc \to \Chom^{\BM,T}(Y)_\loc,
  \]
  since $f$ is a $BG$-torsor for $G=\Ker(T' \twoheadrightarrow T)$.
  As this isomorphism is compatible with proper push-forwards, we find that
  \[
    i_* : \Chom^{\BM,T}(X^{hT'})_\loc \to \Chom^{\BM,T}(X)_\loc
  \]
  is identified with the same map on $\Chom^{\BM,T'}(-)_\loc$.

  It will thus suffice to check the conditions of \thmref{thm:virlocgen} for $i : X^{hT'} \to X$ (where the $T$ in the statement is our $T'$).
  By \propref{Prop:Fixed_locus_DM_is_closed}, it is a closed immersion.
  By \remref{rem:trivact}, it is $T'$-equivariant with respect to the trivial action on $X^{hT'}$.
  By construction of $T' \twoheadrightarrow T$, $X^{hT'}$ contains all geometric points $x$ of $X$ with $\St^T_X(x) = T'_{k(x)}$ (see \corref{cor:iouYVkdqDavArMOFS}).
  We may thus apply \thmref{thm:virlocgen} and \corref{cor:-ph101213} to conclude.
\end{proof}

\section{Simple wall-crossing formula}
\label{sec:wall}

    We prove a wall-crossing formula for simple $\bG_m$-wall crossings as in \cite[\S 2.1, App.~A]{KiemLiWall}, \cite[\S 4]{ChangKiemLi}, \cite[Cor.~2.21]{JoyceEnumerative}.
    In particular, we remove the global resolution assumptions in \emph{op. cit.} (and generalize to Deligne--Mumford stacks over general base fields).

    Let $X$ be a derived Deligne--Mumford stack of finite type over $k$ with $T=\Gm$-action and quotient $\sX=[X/T]$.
    Let $X_+$ and $X_-$ be open substacks of $X$ such that $M_\pm = [X_\pm/T] \sub \sX$ are Deligne--Mumford.
    \begin{defn}
      The \emph{master space} associated to $(X, M_+,M_-)$ is the quotient stack
      \[\mfr{M} = [X\times\P^1\setminus (U_- \times \{0\} \cup U_+ \times \{\infty\})/\Gm]\]
      where $\Gm$ acts diagonally on $X\times\P^1$ and $U_-$ and $U_+$ are the respective complements of $X_+$ and $X_-$ in $X$.
      Note that $\mfr{M}$ is Deligne--Mumford and that the $T$-action on $X$ induces a $T$-action on $\mfr{M}$.
    \end{defn}

    \begin{thm}\label{thm:wall}
      Let $Z = X^{hT'}$ where $T'\twoheadrightarrow T$ is a reparametrization as in \corref{cor:iouYVkdqDavArMOFS}.
      If $X$ is quasi-smooth, then $M_+$, $M_-$ and $Z$ are quasi-smooth and we have
      \[
        [M_+]^\vir - [M_{-}]^\vir \simeq \on{res}_{t = 0} \Big(\frac{[Z]^\vir}{e_T(\cN_{Z/X})}\Big)
      \]
      in $\Chom^{\BM,T}(\mfr{M}^{hT'})_\loc \simeq \CBM(\mfr{M}^{hT'}) \otimes \Ccoh(BT)_\loc$.
    \end{thm}

    \begin{cor}
      Suppose $M_+$, $M_-$ and $Z$ are moreover proper.
      Given $\alpha \in \pi_0\Ccoh_T(X)\vb{d}$, where $X$ is of virtual dimension $d+1$, let
      $$\alpha_{\pm} \in \pi_0\Ccoh_T(M_\pm)\vb{d}$$
      correspond to $\alpha|_{X_\pm} \in \Ccoh_T(X_\pm)\vb{d} \simeq \Ccoh_T(M_\pm)\vb{d}$, where $X_\pm = M_\pm \fibprod_\sX X \sub X$.
      Then we have
      \[
        \alpha_+ \cdot [M_+]^\vir - \alpha_- \cdot [M_{-}]^\vir
        = \on{res}_{t=0} \Big(\alpha \cdot \frac{[Z]^\vir}{e_T(\cN_{Z/X})}\Big).
      \]
    \end{cor}

    Here $t \in \pi_0\Ccoh(BT)\vb{1} \simeq \pi_0\Ccoh(\Spec(k))[t,t^{-1}]$ is the first Chern class of the tautological line bundle, and $\on{res}_{t=0}$ denotes the residue of a Laurent series at $t=0$.

    \begin{proof}[Proof of \thmref{thm:wall}]
      There is a canonical morphism
      \begin{equation}\label{eq:stapled}
        Z \coprod M_+ \coprod M_- \to \mfr{M}^{hT'}
      \end{equation}
      which is clearly bijective (on field-valued points).
      It is also formally étale, as one can see immediately from \corref{cor:anisogamy}.
      Indeed, the relative cotangent complexes over $\mfr{M}$ are given by
      \begin{align*}
        L_{Z/\mfr{M}} &\simeq L_{Z/X} \simeq L_X|_{Z}^\mov[1],\\
        L_{M_+/\mfr{M}} &\simeq L_{M_+/[X_+\times(\P^1\setminus\{\infty\})/\Gm]} \simeq \cO^{(1)}[1],\\
        L_{M_-/\mfr{M}} &\simeq L_{M_-/[X_-\times(\P^1\setminus\{0\})/\Gm]} \simeq \cO^{(-1)}[1],\\
        L_{\mfr{M}^{hT'}/\mfr{M}} &\simeq L_{\mfr{M}}|_{\mfr{M}^{hT'}}^\mov[1],
      \end{align*}
      where $(-)^{(i)}$ indicates the weight of the $\Gm$-action.
      As a radicial étale surjection, \eqref{eq:stapled} is an isomorphism.
      Now it follows from \thmref{thm:qZBC} that we have
      \[
        [\mfr{M}/T]^\vir = \frac{[M_+]^\vir}{-t} + \frac{[M_{-}]^\vir}{t} + \frac{[Z]^\vir}{e_T(\cN_{Z/X})}.
      \]
      The claim follows by applying $\on{res}_{t=0}$.
    \end{proof}

    \begin{rem}
      Following \secref{sec:POT}, one could also formulate the above result using the language of perfect obstruction theories.
      In particular, one may remove the global embeddability or global resolution hypotheses in \cite[Thm.~A.2]{KiemLiWall} and \cite[Thm.~4.2]{ChangKiemLi}.
      See also \cite[Cor.~2.21, Rem.~2.20]{JoyceEnumerative}.
    \end{rem}

    \begin{rem}
      In particular, this proves the non-symmetric analogue of \cite[Conj.~1.2]{KiemLiWall}.
    \end{rem}

\appendix

\section{Fixed loci of group actions on algebraic stacks}
\label{sec:fixed}

In this appendix we fix a scheme $S$ and an fppf group algebraic space $G$ over $S$.
All Artin stacks will be assumed to have separated diagonal.

For an Artin stack $X$ over $S$ with $G$-action, we will define a fixed locus $X^G \sub X$ and study its properties.
We will also introduce a homotopy fixed point stack $X^{hG}$, which usually is not a substack of $X$ but has better deformation-theoretic properties.
In the case of torus actions, we will prove a certain relation between the two constructions (\thmref{thm:boundedness}).

\subsection{Stabilizers of group actions}
\label{ssec:stab}

Given an action of an algebraic group $G$ on an Artin stack $X$, we define the stabilizer of the action (``$G$-stabilizer'') at any point $x$ of $X$.
When the stabilizer at $x$ of $X$ itself is trivial, this coincides with the stabilizer of the quotient stack $[X/G]$ at $x$.

\begin{rem}\label{rem:Yana}
  Let $f : X \to Y$ be a morphism of Artin stacks over $S$.
  The relative inertia stack $I_{X/Y}$ is a group Artin stack over $X$ which fits into a cartesian square
  \[\begin{tikzcd}
    I_{X/Y} \ar{r}\ar{d}
    & I_{X/S} \ar{d}
    \\
    X \ar{r}
    & I_{Y/S} \fibprod_Y X
  \end{tikzcd}\]
  of group stacks over $X$.
  The lower horizontal arrow is the base change of the identity section $e : Y \to I_{Y/S}$.
  When $f$ is representable, $I_{X/Y} \to X$ is an isomorphism, i.e., $I_{X/S} \to I_{Y/S} \fibprod_Y X$ is a monomorphism of group stacks.
  See e.g. \cite[Tag~050P]{Stacks}.
\end{rem}

\begin{rem}
  Let $X$ be an Artin stack over $S$ with $G$-action and denote by $\sX = [X/G]$ the quotient stack.
  Applying \remref{rem:Yana} to the morphisms $X \twoheadrightarrow \sX$ and $\sX \to BG$, we get the cartesian squares of group stacks over $X$
  \[
    \begin{tikzcd}
      X \ar{r}\ar{d}
      & I_{X/S} \ar{r}\ar{d}
      & X \ar{d}
      \\
      X \ar{r}
      & I_{\sX/S} \fibprod_{\sX} X \ar{r}
      & G \fibprod_S X,
    \end{tikzcd}
  \]
  where $I_{X/\sX} \simeq X$ since $X \twoheadrightarrow \sX$ is representable and the right-hand vertical arrow is the identity section.
  For every $S$-scheme $A$ and every $A$-valued point $x$ of $X$, this gives rise to an exact sequence of group algebraic spaces over $A$
  \begin{equation}\label{eq:evvTWDTvdYQvm}
    1 \to \underline{\Aut}_X(x) \to \underline{\Aut}_{\sX}(x)
    \xrightarrow{\alpha_A} G_A
  \end{equation}
  where $G_A := G \fibprod_S A$ denotes the fibre of $G$ over $x$.
\end{rem}

\begin{defn}\label{defn:stabilizer}
  Let $X$ be an Artin stack over $S$ with $G$-action.
  For any scheme $A$ and every $A$-valued point $x$ of $X$, the \emph{$G$-stabilizer} (or \emph{stabilizer of the $G$-action}) at $x$ is an fppf sheaf of groups $\St^G_X(x)$ defined as the cokernel of the homomorphism $\uAut_X(x) \hook \uAut_\sX(x)$.
  Thus we have a short exact sequence
  \begin{equation}\label{eq:stabses}
    1 \to \underline{\Aut}_X(x) \to \underline{\Aut}_{\sX}(x) \to \St^G_X(x) \to 1
  \end{equation}
  of sheaves of groups over $A$.
  Note that $\St^G_X(x)$ can be regarded as a subgroup of $G_{A}$, since it is the image of $\alpha_A : \uAut_{\sX}(x) \to G_{A}$.
\end{defn}

\begin{rem}
  For a field-valued point $x : \Spec(k(x)) \to X$, the $G$-stabilizer $\St^G_X(x)$ is a group algebraic space.
  This follows from \cite[Exp.~V, Cor.~10.1.3]{SGA3}, since in this case $\uAut_X(x)$ is flat over $\Spec(k(x))$.
  Since $X$ has separated diagonal, $\St^G_X(x)$ is moreover a group \emph{scheme} by \cite[0B8F]{Stacks}.
\end{rem}

\begin{rem}\label{rem:noboh1b1}
  When $X$ has trivial stabilizers (i.e., is an algebraic space), then the $G$-stabilizer $\St^G_X(x)$ at a point $x$ is nothing else than the automorphism group $\uAut_{\sX}(x)$ of the quotient stack $\sX = [X/G]$.
\end{rem}

\begin{rem}
  Let $X$ be a \emph{derived} Artin stack locally of finite type over $k$ with $G$-action.
  For any field-valued point $x$ of $X$, the \emph{$G$-stabilizer} of $X$ at $x$ is defined to be the $G$-stabilizer of the classical truncation $X_\cl$ at $x$.
\end{rem}

\begin{rem}\label{rem:zymogenous}
  Let $X$ be an Artin stack over $S$ with $G$-action.
  Let $A$ be an $S$-scheme and $x$ an $A$-valued point of $X$.
  From the short exact sequence \eqref{eq:stabses} we see that the induced morphism $B\uAut_X(x) \to B\uAut_\sX(x)$ is a $\St^G_X(x)$-torsor, where $\sX = [X/G]$.
  Moreover, there is a commutative diagram
  \[\begin{tikzcd}
    B\uAut_X(x) \ar{r}\ar[twoheadrightarrow]{d}
    & X \ar[twoheadrightarrow]{d}
    \\
    B\uAut_\sX(x) \ar{r}
    & \sX
  \end{tikzcd}\]
  where the right-hand arrow is a $G$-torsor.
  In particular, we find that there is a canonical $\St^G_X(x)$-action on the group algebraic space $\uAut_X(x)$, and the canonical monomorphism
  \[ B\uAut_X(x) \hook X \]
  is equivariant with respect to the $\St^G_X(x)$-action on the source and $G$-action on the target.
\end{rem}

\begin{notat}\label{notat:bedark}
  We denote by $\mfr{a}(x)^\vee$ the dual Lie algebra of the group algebraic space $\uAut_X(x)$ over $A$, i.e.,
  \[ \mfr{a}(x)^\vee = e^* \Omega^1_{\uAut_X(x)/A} \]
  where $e : A \to \uAut_X(x)$ is the identity section.
  The $\St^G_X(x)$-action on $\uAut_X(x)$ (\remref{rem:zymogenous}) descends to $\mfr{a}(x)^\vee$.
\end{notat}

\subsection{Fixed points}

\begin{defn}\label{defn:oy1bpy1b1}
  Let $X$ be an Artin stack over $S$ with $G$-action.
  The \emph{$G$-fixed locus} $X^G \sub X$ is the locus where the canonical homomorphism of group algebraic spaces over $X$
  \begin{equation}\label{eq:Obp01bp01}
    I_{\sX} \fibprod_{\sX/S} X \to G \fibprod_S X
  \end{equation}
  is surjective\footnote{%
    i.e., an effective epimorphism of fppf sheaves
  }, where $\sX = [X/G]$.
  In other words, let $A$ be an $S$-scheme and $x \in X(A)$ an $A$-point, and consider the homomorphism of group algebraic spaces over $A$ \eqref{eq:evvTWDTvdYQvm}
  \begin{equation*}
    \alpha_A : \uAut_{\sX}(x) \to G_{A}
  \end{equation*}
  obtained by base changing \eqref{eq:Obp01bp01} along $x : A \to X$.
  Then $x$ belongs to the fixed locus $X^G$ if and only if $\alpha_A$ admits a section after base change along some fppf cover $A' \twoheadrightarrow A$.
\end{defn}

\begin{rem}
  For an $A$-point $x$ of $X$, recall that the image of $\alpha_A$ is the $G$-stabilizer $\St^G_X(x)$ at $x$ (\defnref{defn:stabilizer}).
  Thus, the fixed locus $X^G$ is the locus of points $x$ where the inclusion $\St^G_X(x) \sub G_A$ is an equality.
\end{rem}

\begin{quest}
  Is the inclusion $X^G \to X$ a closed immersion?
\end{quest}

We will see that this holds for algebraic spaces (\propref{prop:ksno1o1}).
It also ``almost'' holds for split torus actions on quasi-DM stacks:

\begin{prop}\label{prop:XTred Artin}
  Suppose $G$ acts on a derived quasi-DM stack $X$ locally of finite type over $k$.
  If $G$ has connected fibres over $S$, then the subset $\abs{X^G} \sub \abs{X}$ is closed.
  In particular, there exists a reduced closed substack $X^G_\red$ of $X$ such that $\abs{X^G_\red} = \abs{X^G}$.
\end{prop}

\begin{proof}
  The subset $\abs{X^G}$ is the locus of points $x \in \abs{X}$ for which $\St^G_X(x)$ is equal to $G_{k(x)} = G \fibprod_S \Spec(k(x))$.
  Since $G_{k(x)}$ is connected (hence irreducible, see \cite[Exp.~VI\textsubscript{A}, Cor.~2.4.1]{SGA3}), this is equivalent to the condition that $\dim(\St^G_X(x)) = \dim(G_{k(x)})$.
  Since $X$ has finite stabilizers, the short-exact sequence \eqref{eq:stabses} shows that $\dim(\St^G_X(x)) = \dim(\uAut_{\sX}(x))$.
  Note that $\uAut_\sX(x)$ is the fibre $\pi^{-1}(x)$ of the projection of the inertia stack $\pi : I_{\sX} \to \sX$ of $\sX=[X/G]$, so closedness of the locus where $\dim(\uAut_{\sX}(x)) = \dim(G_{k(x)})$ follows from the upper semi-continuity of the function $x \mapsto \dim(\pi^{-1}(x))$ on $\abs{\sX}$ (see \cite{RydhSemicontinuous}, \cite[Tag~0DRQ]{Stacks}).
\end{proof}

\begin{defn}\label{defn:XGred}
  Let the notation be as in \propref{prop:XTred Artin}.
  The Artin stack $X^G_\red$ is called the \emph{reduced $G$-fixed locus} of $X$.
  If $X$ is a \emph{derived} quasi-DM stack with $G$-action, then its reduced $G$-fixed locus is the reduced $G$-fixed locus of $X_\cl$ (with the induced $G$-action).
\end{defn}

\subsection{Homotopy fixed points}
\label{ssec:hofix}

For classical stacks, the following definition is studied in \cite{RomagnyGroupActions,RomagnyFixedMult,FixedLocusRomagny}.

\begin{defn}\label{defn:yfgyoav1}
  Let $X$ be a derived Artin stack over $S$ with $G$-action.
  The \emph{homotopy fixed point stack} of $X$ is the stack of $G$-equivariant morphisms $S \to X$, i.e.,
  \[
    X^{hG} := \uMap^G_S(S, X),
  \]
  where $S$ is regarded with trivial $G$-action.
\end{defn}

\begin{rem}\label{rem:trivact}
  For every derived stack $A$ over $S$, we have by definition a canonical isomorphism
  \[
    \Maps_S(A, X^{hG})
    \simeq \Maps_S(A, \uMap^G_S(S, X))
    \simeq \Maps^G_S(A, X)
  \]
  where $A$ is regarded with trivial $G$-action.
  For example, the identity $X^{hG} \to X^{hG}$ corresponds to a canonical $G$-equivariant $S$-morphism
  \begin{equation*}
    \varepsilon := \varepsilon_X^G : X^{hG} \to X,
  \end{equation*}
  where $X^{hG}$ is regarded with trivial $G$-action.
\end{rem}

\begin{rem}\label{rem:gfp01b}
  Equivalently, $X^{hG}$ can be described as the Weil restriction of $[X/G] \to BG$ along $BG \to S$.
  In other words, it classifies sections of $[X/G] \to BG$ and fits into the homotopy cartesian square
  \[
    \begin{tikzcd}
      X^{hG} \ar{r}\ar{d}
      & \uMap_S(BG, [X/G]) \ar{d}
      \\
      S \ar{r}{\id_{BG}}
      & \uMap_S(BG, BG).
    \end{tikzcd}
  \]
\end{rem}

\begin{rem}\label{rem:onoyb01b1}
  Yet another way to describe $X^{hG}$ is that it classifies group-theoretic sections of the homomorphism of group algebraic spaces over $X$
  \[
    I_\sX \fibprod_\sX X \to G \fibprod_S X
  \]
  where $\sX = [X/G]$.
  Indeed, there is a cartesian square
  \[\begin{tikzcd}
    X^{hG} \ar[swap]{d}{\varepsilon}\ar{r}
    & \uGrp_X(G\fibprod_S X, I_\sX \fibprod_\sX X) \ar{d}
    \\
    X \ar{r}{\id}
    & \uGrp_X(G\fibprod_S X, G\fibprod_S X)
  \end{tikzcd}\]
  where $\uGrp_X(-,-)$ denotes the sheaf of group homomorphisms over $X$.
  See e.g. \cite[Lem.~4.1.2]{FixedLocusRomagny}.
  In terms of points, we see that for any $S$-scheme $A$ a lift of an $A$-point $x$ of $X$ along $\varepsilon : X^{hG} \to X$ amounts to a \emph{group-theoretic section} of the homomorphism \eqref{eq:evvTWDTvdYQvm}
  $$\alpha_A : \uAut_{\sX}(x) \to G_{A}$$
  of group algebraic spaces over $A$.
  Comparing with \defnref{defn:oy1bpy1b1}, we find in particular that $\varepsilon: X^{hG} \to X$ factors through the fixed locus $X^G \sub X$.
\end{rem}

\begin{rem}
  Informally speaking, we can think of a point of $X^{hG}$ as a point $x$ of $X$ together with a collection of ``fixings'', i.e., for every point $g$ of $G$ an isomorphism $g \cdot x \simeq x$, together with a homotopy coherent system of compatibilities between them (with respect to the group operation).
\end{rem}

For the next statement, we recall the notion of \emph{formal properness} from \cite{HalpernLeistnerPreygel}.

\begin{rem}\label{rem:bigwiggery}
  The classifying stack $BG$ is formally proper over $S$ when either
  \begin{inlinelist}
    \item
    $S$ is the spectrum of a field and $G$ is reductive (see \cite[Ex.~4.3.5]{HalpernLeistnerPreygel}); or
    \item
    $S$ is noetherian and $G$ is linearly reductive (see Prop.~4.2.3 and Thm.~4.2.1 in \cite{HalpernLeistnerPreygel}).
  \end{inlinelist}
\end{rem}

\begin{thm}[Halpern-Leistner--Preygel]\label{thm:collegial}
  Let $X$ be a derived Artin stack locally of finite type over $S$ with $G$-action.
  Assume $BG$ is formally proper over $S$.
  If $X$ is $1$-Artin with affine stabilizers, then the derived stack $X^{hG}$ is $1$-Artin with affine stabilizers.
\end{thm}
\begin{proof}
  Follows from \cite[Thm.~5.1.1, Rmk.~5.1.3]{HalpernLeistnerPreygel} in view of \remref{rem:gfp01b}.
\end{proof}

For $G$ of multiplicative type (and $X$ classical), a different proof of \thmref{thm:collegial} was given in \cite[Thm.~1]{RomagnyFixedMult}.

\ssec{Properties}

We record some generalities about the constructions $X^G$ and $X^{hG}$.
Our main interest is in the properties of the canonical morphisms $X^G \hook X$ and $\varepsilon : X^{hG} \to X$.

We begin by comparing $X^G$ with the classical truncation of $X^{hG}$ in the case of algebraic spaces.
This is probably well-known.
We will see in \corref{cor:roasting} below that when $G$ is linearly reductive and $X$ is smooth, we moreover have $X^{hG} \simeq X^{hG}_\cl$.

\begin{prop}\label{prop:ksno1o1}
  Let $X$ be an algebraic space over $S$ with $G$-action.
  If $G$ has connected fibres, then there is a canonical isomorphism $X^G \simeq (X^{hG})_\cl$ over $X$.
  Moreover, the morphisms $X^G \to X$ and $\varepsilon : X^{hG} \to X$ are closed immersions.
\end{prop}
\begin{proof}
  By \cite[Thm.~A.1]{RomagnyApp}, $(X^{hG})_\cl \to X$ is a closed immersion, hence so is $\varepsilon : X^{hG} \to X$ (this property can be checked on classical truncations).
  Let us show that the canonical morphism (\remref{rem:onoyb01b1})
  \begin{equation}\label{eq:glossalgy}
    X^{hG}_\cl \to X^{hG} \to X^G
  \end{equation}
  is invertible.
  Since $X^G \to X$ and $(X^{hG})_\cl \to X$ are both monomorphisms, it will suffice to show that \eqref{eq:glossalgy} is surjective on $A$-valued points (for all $S$-schemes $A$).
  But since $X$ has trivial stabilizers, for every $A$-valued point $x : A \to X$ the canonical homomorphism $\uAut_{[X/G]}(x) \to G_A$ is surjective if and only if it is invertible (see \remref{rem:noboh1b1}).
  In particular, if $x$ belongs to $X^G$ then $\uAut_\sX(x) \to G_A$ already admits a section over $A$.
\end{proof}

Let us now turn our attention to the morphism $\varepsilon : X^{hG} \to X$.
In general, it is not a closed immersion or even a monomorphism (\examref{exam:uf0pb1hp1}).
We begin with the following statement (we thank M.~Romagny for providing the idea for the proof).

\begin{prop}\label{prop:bowermaiden}
  Suppose $S$ is locally noetherian and $G$ has smooth and connected fibres over $S$.
  Let $X$ be a derived quasi-DM stack with $G$-action, which is locally of finite type over $S$.
  Then the canonical morphism $\varepsilon : X^{hG} \to X$ is essentially proper\footnote{%
    I.e., it is locally of finite type and satisfies the valuative criterion for properness.
    Thus essentially proper $+$ quasi-compact $\Rightarrow$ proper.
    See \cite[IV\textsubscript{4}, Rem.~18.10.20]{EGA}.
  }.
\end{prop}
\begin{proof}
  Since $\varepsilon$ is separated and locally of finite presentation \cite{RomagnyFixedMult,FixedLocusRomagny}, it will suffice to show that for every discrete valuation ring $R$ with fraction field $K$ and every commutative solid arrow diagram
  \[\begin{tikzcd}
    \Spec(K) \ar{r} \ar{d} & X^{hG} \ar{d}{\varepsilon}\\
    \Spec(R) \ar{r} \ar[dashed]{ru} & X,
  \end{tikzcd}\]
  there exists a dashed arrow making the total diagram commute.
  This amounts to showing that for every $R$-point $x$ of $X$ and any group-theoretic section $\sigma_K$ of the morphism $\alpha_K : \uAut_\sX(x_K) \to G_K$ \eqref{eq:evvTWDTvdYQvm}, there exists a section
  \[\sigma : G \to \uAut_\sX(x)\]
  of $\alpha_R : \uAut_\sX(x) \to G_R$ which lifts $\sigma_K$.

  Since $X$ has finite stabilizers, $\alpha_R$ is quasi-finite (since its kernel $\uAut_X(x)$ is quasi-finite).
  Since $\alpha_R$ is affine (as a morphism between affine schemes, by \cite[IX, Lem.~2.2]{RaynaudFaisceaux}), the section $\sigma_K$ is a closed immersion.
  Let $\Gamma \sub \uAut_\sX(x)$ denote the schematic closure of $\sigma_K(G_K) \sub \uAut_\sX(x_K)$ in $\uAut_\sX(x)$.
  This is a closed subgroup of $\uAut_\sX(x)$ (see \cite[\S 4.1]{EffectiveModels}).
  We claim that the homomorphism of group $R$-schemes
  \begin{equation}\label{eq:manna}
    \Gamma \sub \uAut_\sX(x) \to G_R
  \end{equation}
  is invertible.
  Since it is quasi-finite, separated and birational (because over $K$ it is the isomorphism $\Gamma_K \sub \uAut_\sX(x_K) \to G_K$) with normal target, Zariski's main theorem implies that \eqref{eq:manna} is an open immersion.

  Let $\mfr{m} \sub R$ denote the maximal ideal.
  The base change $\alpha_{R/\mfr{m}}$ has \emph{finite} kernel, hence is finite.
  Thus \eqref{eq:manna} base changes to a finite open immersion (i.e., an inclusion of connected components) over $R/\mfr{m}$.
  Since $G_{R/\mfr{m}}$ is irreducible and $\Gamma_{R/\mfr{m}}$ is nonempty, this shows that \eqref{eq:manna} is bijective over $R/\mfr{m}$.
  Since it is also an isomorphism over the fraction field $K$, it follows that \eqref{eq:manna} is invertible as claimed.

  We now obtain the desired section $\sigma$ by taking the composite
  \[
    \sigma : G_R \xleftarrow{\sim} \Gamma \sub \uAut_\sX(x)
  \]
  which is a group homomorphism by construction.
\end{proof}

For torus actions on Deligne--Mumford stacks, we see that $\varepsilon$ is a closed immersion:

\begin{prop}
\label{Prop:Fixed_locus_DM_is_closed}
  Assume that $G=T$ is a torus and $S$ is locally noetherian.
  Let $X$ be a derived Deligne--Mumford stack with $T$-action, quasi-separated and locally of finite type over $S$ with separated diagonal.
  Then the canonical morphism $\varepsilon : X^{hT} \to X$ is a closed immersion.
\end{prop}
\begin{proof}
  We may assume that $X$ is classical.
  By \propref{prop:bowermaiden}, it is enough to show that $\varepsilon$ is a locally closed immersion.
  For this we may argue as in the proof of \cite[Thm.~5.16]{AlperHallRydhLuna}\footnote{%
    Note that \emph{loc. cit.} claims to show that $\varepsilon$ is a closed immersion, but in fact the argument only shows it is locally closed due to the use of their generalized Sumihiro theorem.
    We thank A.~Kresch for pointing this out to us.
  } to reduce to the case where $S$ is the spectrum of an algebraically closed field and $X$ is affine (note that the first statement of \cite[Prop~5.20]{AlperHallRydhLuna} only uses connectedness of the group).
  Then the claim follows from \propref{prop:ksno1o1}.
\end{proof}

\begin{rem}
  If $G$ is of multiplicative type and $X$ is $1$-Artin with affine and finitely presented diagonal, then $\varepsilon : X^{hG} \to X$ is representable, separated and locally of finite presentation (see \cite[Thm.~1]{RomagnyFixedMult}).
  This is generalized further in \cite{FixedLocusRomagny}.
\end{rem}

\begin{exam}[Romagny]\label{exam:uf0pb1hp1}
  If $G$ acts on a quasi-DM stack $X$, the canonical morphism $\varepsilon: X^{hG} \to X$ is not generally a monomorphism or even unramified.
  The following example appears as \cite[Ex.~4.1.5]{FixedLocusRomagny}, where it is shown that $\varepsilon$ is not a monomorphism.
  Here we show that in the same example, $\varepsilon$ is not even unramified.

  Let $S$ be the spectrum of an algebraically closed field $k$ of characteristic $p>0$ and let $G=T$ be the rank one torus $\bG_{m,k}$.
  Let $\alpha_p$ denote the group $k$-scheme fitting in the short-exact sequence
  \[
    0 \to \alpha_p \to \bG_{a,k} \xrightarrow{F} \bG_{a,k} \to 0
  \]
  where $F$ sends $x \mapsto x^p$.
  The latter is $\bG_m$-equivariant with respect to the scaling action with weight $1$ on the source and weight $p$ on the target, so $\alpha_p$ inherits a $\bG_m$-action by scaling.
  This gives rise to a $\Gm$-action on the classifying stack $X = B\alpha_p$, and we claim that the morphism $\varepsilon : X^{h\Gm} \to X$ is ramified.

  By \corref{cor:xiphias} it will suffice to show that the dual Lie algebra $\mfr{a}(x)^\vee$ of $\uAut_X(x)$ has nonzero moving part, where $x : \Spec(k) \twoheadrightarrow X = B\alpha_p$ is the quotient morphism.
  But $\uAut_X(x) = \alpha_p$ and $\mfr{a}(x)^\vee$ is the Lie algebra of $\alpha_p$, which is $H^0$ of the cotangent complex of $\alpha_p$ (restricted along the identity section), which since $dF = 0$ we compute ($\Gm$-equivariantly) as
  \[
    \cO^{(-1)} \oplus \cO^{(-p)}[1]
  \]
  where $\cO^{(i)}$ is the structure sheaf with weight $i$ scaling action.
\end{exam}

\subsection{Deformation theory of homotopy fixed points}

All cotangent complexes are relative to the base $S$ unless specified otherwise.

\begin{rem}
  When $BG$ is formally proper over $S$ (see \remref{rem:bigwiggery}), the functor $f^* : \Dqc(S) \to \Dqc(BG)$ of inverse image along $f : BG \to S$ admits a \emph{left} adjoint
  \[ f_+ : \Dqc(BG) \to \Dqc(S), \]
  see \cite[Prop.~5.1.6]{HalpernLeistnerPreygel} which is computed by $f_+(\cF) = f_*(\cF^\vee)^\vee$ on perfect complexes $\cF$.
  The same holds for any base change of $f$.
  Under the identification $\Dqc(BG) \simeq \Dqc^G(S)$, $f_*$ and $f_+$ are the functors of (derived) $G$-invariants and coinvariants, respectively.
\end{rem}

\begin{cor}\label{Cor:cotangentcplxforfixedlocus}
  Suppose that $BG$ is formally proper over $S$.
  Then the cotangent complex of $X^{hG}$ is given by
  \begin{equation*}
    L_{X^{hG}} \simeq \bar{f}_+ e^*(L_{[X/G]/BG})
  \end{equation*}
  where $e : X^{hG} \times BG \simeq [X^{hG}/G] \to [X/G]$ is the quotient of $\varepsilon : X^{hG} \to X$ and $\bar{f} : X^{hG} \fibprod_S BG \to X^{hG}$ is the projection.
\end{cor}
\begin{proof}
  This is the formula for the cotangent complex of a Weil restriction given in \cite[Prop.~19.1.4.3]{LurieSAG}, which generalizes to our formally proper situation as in \cite[Prop.~5.1.10]{HalpernLeistnerPreygel}.
\end{proof}

In other words, \corref{Cor:cotangentcplxforfixedlocus} states that $L_{X^{hG}}$ is given by the (derived) $G$-coinvariants of the pull-back to $X^{hG}$ of the cotangent complex $L_{X}$ (regarded with its canonical $G$-action).
Dually, the tangent complex is given by the (derived) $G$-invariants (= $G$-fixed part) of the pull-back to $X^{hG}$ of the tangent complex $T_{X}$.
In the linearly reductive case, we do not need to distinguish between invariants and coinvariants.

\begin{lem}\label{lem:palaeophytologist}
  Suppose $G$ is linearly reductive.
  Let $X$ be a locally noetherian Artin stack over $S$ and write $f : X \fibprod_S BG \to X$ for the projection.
  Then there is a canonical isomorphism
  \[ f_*(\cF) \to f_+(\cF) \]
  for every quasi-coherent complex $\cF \in \Dqc(X \fibprod_S BG)$.
\end{lem}
\begin{proof}
  We will show that the canonical morphism
  \[
    f^*f_*(\cF)
    \xrightarrow{\mrm{counit}} \cF
    \xrightarrow{\mrm{unit}} f^*f_+(\cF)
  \]
  is invertible; the claimed isomorphism will then follow by applying $*$-inverse image along the quotient morphism $S \twoheadrightarrow BG$.
  We may simplify notation by replacing $G$ by its base change $G \fibprod_S X$.
  Since all functors involved commute with $*$-inverse image ($f_*$ satisfies base change because $f$ is universally of finite cohomological dimension, and $f_+$ satisfies base change by adjunction, see e.g. \cite[Lem.~5.1.8]{HalpernLeistnerPreygel}), we may use fpqc descent to reduce to the case where $X$ is a noetherian scheme.
  Since $*$-inverse image to residue fields is jointly conservative (by noetherianness), we may then further assume that $X$ is the spectrum of a field $k$.
  Since $G$ is linearly reductive (and embeddable), $BG$ is perfect (see e.g. \cite[Thm.~1.42]{KhanKstack}) so we may assume that $\cF$ is a perfect complex (again, $f_*$ preserves colimits because $f$ is universally of finite cohomological dimension).
  Note that $f^*$ is t-exact (since $f$ is flat), $f_*$ is t-exact (since $G$ is linearly reductive), and $f_+$ is t-exact on perfect complexes (since $f_+(-) \simeq f_*(-^\vee)^\vee$).
  Thus we may also reduce to the case where $\cF$ is a (discrete) coherent sheaf.
  In other words, we are reduced to showing that for every finite-dimensional vector space $V$ over $k$ with $G$-action, the canonical morphism $V^G \sub V \twoheadrightarrow V_G$ (from $G$-invariants to $G$-coinvariants) is invertible, which is well-known (for example, this follows easily from the characterization of linear reductivity in \cite[Prop.~12.6(vi)]{AlperGood}).
\end{proof}

\begin{defn}\label{defn:bodach}
  Suppose $G$ is linearly reductive.
  Let $X$ be a locally noetherian Artin stack over $S$ (with trivial $G$-action).
  Given a quasi-coherent sheaf $\cF \in \Dqc(X \fibprod_S BG) \simeq \Dqc^G(X)$, the \emph{fixed part} of $\cF$ is defined as $\cF^\fix := f^*f_*(\cF)$ and the \emph{moving part} of $\cF$ is the cofibre of the canonical morphism $\cF^\fix \to \cF$.
  By \lemref{lem:palaeophytologist} the latter admits a (natural) retraction $\cF \to f^*f_+(\cF) \simeq f^*f_*(\cF) = \cF^\fix$.
  Thus the exact triangle
  \[
    \cF^\fix \to \cF \to \cF^\mov
  \]
  is split, i.e., there are canonical isomorphisms $\cF \simeq \cF^\fix \oplus \cF^\mov$, natural in $\cF$.
  (These definitions are compatible with \defnref{defn:uh1pbpbdf} in case $G$ is diagonalizable.)
\end{defn}

\begin{cor}\label{cor:roasting}
  Suppose $G$ is linearly reductive.
  Let $X$ be a locally noetherian Artin stack over $S$ with $G$-action.
  There is a canonical isomorphism
  \begin{equation*}
    L_{X^{hG}} \simeq (\varepsilon^* L_X)^\fix
  \end{equation*}
  in $\Dperf(X^{hG})$.
  In particular, if $L_X$ is perfect of Tor-amplitude $\le n$, then so is $L_{X^{hG}}$; if $X$ is smooth (resp. quasi-smooth) over $S$, then so is $X^{hG}$.
\end{cor}
\begin{proof}
  Since $G$ is linearly reductive, the functor of derived $G$-invariants is t-exact.
  Thus the claim follows from \corref{Cor:cotangentcplxforfixedlocus}.
\end{proof}

\begin{cor}\label{cor:anisogamy}
  Suppose $G$ is linearly reductive.
  Let $X$ be a locally noetherian Artin stack over $S$ with $G$-action.
  There is a canonical identification of exact triangles in $\Dqc^G(X^{hG}) \simeq \Dqc(X^{hG} \fibprod_S BG)$
  \[\begin{tikzcd}
    \varepsilon^* L_X \ar{r}\ar[equals]{d}
    & L_{X^{hG}} \ar{r}\ar[equals]{d}
    & L_{X^{hG}/X} \ar[equals]{d}
    \\
    \varepsilon^* L_X \ar{r}
    & (\varepsilon^* L_X)^\fix \ar{r}
    & (\varepsilon^* L_X)^\mov[1].
  \end{tikzcd}\]
\end{cor}

\begin{cor}\label{cor:xiphias}
  Suppose $G$ is linearly reductive.
  Let $X$ be a locally noetherian Artin stack over $S$ with $G$-action.
  Then the morphism $\varepsilon : X^{hG} \to X$ is formally unramified if and only if, for every geometric point $x$ of $X^{hG}$, the dual Lie algebra $\mfr{a}(x)^\vee$ of $\uAut_X(x)$ has vanishing moving part (with respect to the $G$-action defined in \notatref{notat:bedark}\footnote{%
    By abuse of notation, we identify $x$ with its image $\varepsilon(x)$ in $X$.
    Since the latter belongs to the fixed locus $X^G$ (\remref{rem:onoyb01b1}), the $\St^G_X(x)$-action defined in \notatref{notat:bedark} is a $G$-action.
  }).
\end{cor}
\begin{proof}
  Recall that formal unramifiedness is the condition that $\H^0 (L_{X^{hG}/X}) = \Omega^1_{X_\cl^{hG}/X_\cl}$ vanishes.
  We may therefore replace $X$ by its classical truncation.
  By \remref{rem:zymogenous}, the canonical monomorphism $B\uAut_X(x) \hook X$ is $G$-equivariant.
  The relative cotangent complex of the latter vanishes, so that there is a canonical isomorphism
  \[ x^* L_X \simeq x^* L_{B\uAut_X(x)} \simeq \mfr{a}(x)^\vee[-1] \]
  in $\Dqc^G(\Spec(k(x)))$ (where $x$ also denotes the morphism $\Spec(k(x)) \to B\uAut_X(x)$ by abuse of notation).
  By \corref{cor:anisogamy} we get a canonical isomorphism
  \[
    x^* L_{X^{hG}/X}
    \simeq x^* (L_X)^\mov[1]
    \simeq (\mfr{a}(x)^\vee)^\mov,
  \]
  whence the claim.
\end{proof}

Note that \corref{cor:roasting} can be generalized as follows:

\begin{lem}\label{lem:barkpeel}
  Suppose $G$ is linearly reductive.
  Let $X$ be a locally noetherian Artin stack over $S$ with $G$-action.
  Let $i : Z \to X$ be a $G$-equivariant finite unramified morphism such that:
  \begin{defnlist}
    \item\label{item:bojite}
    The $G$-action on $Z$ is trivial.

    \item\label{item:euthytropic}
    The conormal sheaf $\cN_{Z/X} \in \Dqc^{G}(Z) \simeq \Dqc(Z \times BG)$ has no $G$-fixed part.
  \end{defnlist}
  Then the canonical morphism $i^*L_X \to L_Z$ in $\Dqc^{G}(Z) \simeq \Dqc(Z \times BG)$ induces an isomorphism $L_Z \simeq (L_Z)^\fix \simeq (i^*L_X)^\fix$.
  In particular, if $X$ is smooth (resp. quasi-smooth) then so is $Z$.
\end{lem}
\begin{proof}
  The first isomorphism is because $L_Z$ has no moving part by \itemref{item:bojite}, and the second is because the cofibre $L_{Z/X}$ has no fixed part by \itemref{item:euthytropic}.
\end{proof}

\subsection{Reparametrized homotopy fixed points for torus actions}

\begin{rem}\label{rem:vhVyCt}
  Let $\rho : G' \twoheadrightarrow G$ be a surjective homomorphism between group schemes over $S$.
  Given an $S$-scheme $A$ and an $S$-morphism $x : A \to X^{hG'}$, consider the commutative square
  \begin{equation}\label{eq:fnb0p1bp}
    \begin{tikzcd}
      \uAut_{[X/G']}(x) \ar[twoheadrightarrow]{r}\ar[d]
      & G'_A \ar[twoheadrightarrow]{d}{\rho}
      \\
      \uAut_{[X/G]}(x) \ar{r}
      & G_A
    \end{tikzcd}
  \end{equation}
  of group schemes over $A$.

  \begin{defnlist}
    \item
    Note that the square is cartesian.
    Indeed, the induced map on kernels of the horizontal maps may be identified with the identity of $\uAut_X(x)$.

    \item
    Since the upper horizontal and right-hand vertical arrows are surjective, the same holds for the lower arrow.
    This shows that $x$ factors through the fixed locus $X^G$.
    Allowing $x$ to vary, we see that the canonical morphism $\varepsilon_X^{G'} : X^{hG'} \to X$ factors through
    \begin{equation}\label{eq:cellepore}
      X^{hG'} \to X^G.
    \end{equation}
  \end{defnlist}
\end{rem}

\begin{defn}
  Let $T$ be a split torus over $S$.
  A \emph{reparametrization} of $T$ is an isogeny $\rho : T' \twoheadrightarrow T$ where $T'$ is a split torus.
  A morphism of reparametrizations $T'' \to T'$ is a morphism over $T$.
\end{defn}

\begin{rem}\label{rem:ferial}
  Note that the category of reparametrizations of $T$ is filtered.
  Indeed, say $T$ is of rank $r$.
  Reparametrizations $T' \twoheadrightarrow T$ are in bijection with diagonal $(r\times r)$-matrices $(d_1, \ldots, d_r)$, $d_i\in\bZ\setminus \{0\}$.
  Given two reparametrizations $T' \twoheadrightarrow T$ and $T'' \twoheadrightarrow T$, corresponding to tuples $(d'_1, \ldots, d'_r)$ and $(d''_1, \ldots, d''_r)$, there is a morphism $T'' \to T'$ if and only if each $d''_i$ divides $d'_i$.
  Thus the category is equivalent to a poset.
\end{rem}

\begin{defn}\label{defn:XrhT}
  Let $G=T$ be a split torus over $S$.
  Let $X$ be an Artin stack over $S$ with $T$-action.
  The \emph{reparametrized homotopy fixed point stack} $X^{rhT} \to X$ is defined as the filtered colimit
  \[ X^{rhT} := \colim_{\rho} X^{hT'} \]
  over all reparametrizations $\rho: T' \twoheadrightarrow T$.
  The morphisms $\varepsilon_X^{T'} : X^{hT'} \to X$ induce a canonical morphism $\varepsilon^{rT}_X : X^{rhT} \to X$, $T$-equivariant with respect to the trivial action on $X^{rhT}$.
\end{defn}

\begin{prop}\label{prop:cobang}
  Let $G=T$ be a split torus of rank $r$ over $S$ acting on a quasi-separated Deligne--Mumford stack $X$ locally of finite presentation over $S$.
  \begin{thmlist}
    \item\label{item:amphoral}
    For any reparametrization $T' \twoheadrightarrow T$, the induced map
    \begin{equation}\label{eq:rJWZGjGzRDBKVOsFM}
      X^{hT} \to X^{hT'}
    \end{equation}
    is an open and closed immersion.

    \item\label{item:drew}
    There is a canonical decomposition
    \[
      X^{rhT} = \bigsqcup_{\rho} X^{rhT}_\rho
    \]
    over reparametrizations $\rho : T' \twoheadrightarrow T$, where $X^{rhT}_\rho$ is the open and closed substack
    \[ X^{rhT}_\rho = X^{hT'} \setminus \bigcup_{\rho' : T'' \twoheadrightarrow T} X^{rhT''} \]
    where $\rho'$ varies over reparametrizations that factor $\rho$ via some non-identical reparametrization $T' \twoheadrightarrow T''$.

    \item\label{item:ironclad}
    If $X$ is quasi-compact, then $X^{rhT}$ is quasi-compact.
    In particular, we have $X^{rhT} = X^{hT'}$ for some reparametrization $\rho : T' \twoheadrightarrow T$.
  \end{thmlist}
\end{prop}

\begin{proof}
  \itemref{item:amphoral}:
  For any reparametrization $T' \twoheadrightarrow T$, the induced map
  \begin{equation*}
    X^{hT} \to X^{hT'}
  \end{equation*}
  is formally étale, since the relative cotangent complex vanishes by \corref{cor:anisogamy}.
  Moreover, \eqref{eq:rJWZGjGzRDBKVOsFM} is locally of finite presentation, hence étale, at least if $X$ has affine and finitely presented diagonal, see \cite{RomagnyFixedMult}.
  If $X$ is quasi-separated, Deligne--Mumford and locally of finite presentation over $S$, then \eqref{eq:rJWZGjGzRDBKVOsFM} is also a closed immersion (\propref{Prop:Fixed_locus_DM_is_closed}).
  Thus in that case it is an open and closed immersion.

  \itemref{item:drew}: follows from \itemref{item:amphoral}.

  \itemref{item:ironclad}:
  Suppose $X$ is quasi-compact.
  The closed substacks $X^{hT'}$ stabilize as $\rho : T'\twoheadrightarrow T$ varies among reparametrizations.
  Indeed, recall that each $X^{hT'}$ is a closed substack of $X$ (\propref{Prop:Fixed_locus_DM_is_closed}) and the colimit $\colim_\rho X^{hT'}$ over reparametrizations $\rho : T' \twoheadrightarrow T$ is also closed in $X$ by \thmref{thm:boundedness} and \propref{prop:XTred Artin}.
  In particular, it is quasi-compact because $X$ is quasi-compact.
\end{proof}

\begin{rem}\label{rem:noncottager}
  We have the following more explicit description of the open and closed substack $X^{rhT}_\rho$, for a given reparametrization $\rho : T' \twoheadrightarrow T$.
  Let $x$ be a geometric point of $X^{rhT}$.
  Then we have $x \in X^{rhT}_\rho$ if and only if $\rho$ is identified with the reparametrization $\alpha_A : \uAut_{\sX}(x)^0_\red \to T_A$, where the source is the reduced identity component of $\uAut_{\sX}(x)$.
  Indeed, the group-theoretic section of $\uAut_{[X/T']}(x) \to T'$ gives rise to a group homomorphism $T' \to \uAut_{[X/T]}(x)$ from a reduced and connected algebraic group, which thus factors through $\uAut_{[X/T]}(x)^0_\red$.
\end{rem}

\begin{rem}
  For $T$ a rank one torus over a field acting on a Deligne--Mumford stack, the reparametrized homotopy fixed stack $X^{rhT}$ is used as the definition of $T$-fixed locus in \cite[Def.~5.25]{AlperHallRydhLuna} (see also \cite[Prop.~5.3.4]{Kresch}).
\end{rem}

\subsection{Fixed vs. reparametrized homotopy fixed}

  In this subsection our goal is to compare the fixed locus $X^G$ with the homotopy fixed point stack $X^{hG}$.
  The two constructions are typically different, as $\varepsilon : X^{hG} \to X$ may not be a monomorphism (\examref{exam:uf0pb1hp1}).
  But even if we just consider the set-theoretic image of $X^{hG}$ in $\abs{X}$, it will typically not coincide with $\abs{X^G}$.
  In the case where $G=T$ is a torus, we can somewhat bridge the gap by replacing $X^{hT}$ by its reparametrized version $X^{rhT}$.
  In this case, we will prove:

  \begin{thm}
  \label{thm:boundedness}
    Suppose $G=T$ is a split torus acting on a $1$-Artin stack $X$ over $S$ with affine stabilizers.
    Then the morphisms \eqref{eq:cellepore} induce a canonical surjection
    \[ X^{rhT} \twoheadrightarrow X^T \]
    over $X$.
  \end{thm}

  \begin{proof}
    As $\rho : T' \twoheadrightarrow T$ varies, the canonical morphisms $X^{hT'} \to X^T$ \eqref{eq:cellepore} are compatible by construction.
    Thus there is a canonical morphism
    \[ X^{rhT} \to X^T. \]

    For surjectivity, let $x$ be a geometric point of $X^T$.
    Then we have the canonical surjection $\uAut_{[X/T]}(x) \twoheadrightarrow T_{k(x)}$ \eqref{eq:stabses}.
    By \cite[Cor.~1 of Prop.~11.14]{Borel}, there exists a (split) subtorus of $\uAut_{[X/T]}(x)$ on which this restricts to an isogeny.
    By \cite[Exp.~VIII, Cor.~1.6]{SGA3}, we may lift this to an isogeny $\rho : T' \to T$ over $S$.

    Using the cartesian square \eqref{eq:fnb0p1bp} (taking now $A=\Spec(k(x))$), we see that in order to lift $x$ to $X^{hT'}$ it is enough to show that $\uAut_{[X/T]}(x) \to T_{k(x)}$ becomes surjective after base change along $\rho$.
    But $\rho$ factors through $\uAut_{[X/T]}(x) \to T_{k(x)}$ by construction, so this is clear.
  \end{proof}

  \begin{cor}\label{cor:iouYVkdqDavArMOFS}
    Suppose $G=T$ is a split torus acting on a quasi-compact $1$-Artin stack $X$ over $S$ with affine stabilizers.
    Then there exists a reparametrization $\rho : T' \twoheadrightarrow T$ such that the morphism \eqref{eq:cellepore} induces a surjection
    \[ X^{hT'}_\red \twoheadrightarrow X^T_\red \]
    of reduced $1$-Artin stacks (where the target is the reduced $T$-fixed locus, see \defnref{defn:XGred}).
    In particular, the set-theoretic image of $\varepsilon : X^{hT'} \to X$ coincides with $\abs{X^T} \sub \abs{X}$.
  \end{cor}

\subsection{Edidin--Rydh fixed vs. reparametrized homotopy fixed}

\begin{defn}\label{defn:Zaglossus}
  Let $G=T$ be a split torus of rank $r$ over $S$ acting on a quasi-DM stack $X$ over $S$.
  We define $X^{sT}$ as the stack over $X$ whose $A$-valued points, for an $X$-scheme $x : A \to X$, are closed subgroups of $\uAut_{[X/T]}(x)$ which are affine and smooth over $A$ with connected fibres of dimension $r$.
\end{defn}

\begin{rem}\label{rem:faraway}
  The definition of $X^{sT}$ is a variant of the construction in \cite[Prop.~C.5]{EdidinRydh} of, for a $1$-Artin stack $\sX$, a stack $\sX^\max \to \sX$ that can be thought of as the locus of points with maximal-dimensional stabilizer.
  Whenever $X^{sT} \ne \initial$, then we have
  \[
    X^{sT} = \sX^{\max} \fibprod_{\sX} X
  \]
  where $\sX = [X/T]$.
\end{rem}

In this subsection we will prove:

\begin{thm}\label{thm:fotch}
  Let $G=T$ be a split torus over $S$.
  Let $X$ be a tame Deligne--Mumford stack which is quasi-separated and locally of finite presentation over $S$ with $T$-action.
  Then there is a canonical isomorphism
  \[ X^{rhT} \simeq X^{sT} \]
  over $X$.
\end{thm}

Since $X^{rhT}$ is a closed substack of $X$ (see \propref{prop:cobang}), \thmref{thm:fotch} shows that $X^{sT}$ is also a closed substack of $X$ and in particular is Deligne--Mumford.
When $X$ is noetherian and $[X/T]$ admits a good moduli space in the sense of \cite{AlperGood}, this follows from \cite[Prop.~C.5]{EdidinRydh}.
We begin with the following generalization to our situation:

\begin{thm}\label{thm:inarch}
  Let $X$ be a $1$-Artin stack over $S$ with finite diagonal and tame stabilizers and $T$-action.
  Then the morphism $X^{sT} \to X$ is a closed immersion.
  In particular, $X^{sT}$ is $1$-Artin with finite diagonal.
\end{thm}

\begin{proof}
  Note that $X^{sT}$ is stable under base change by étale representable $T$-equivariant morphisms $p : Y \to X$, i.e., the induced morphism
  \[
    Y^{sT} \to X^{sT} \fibprod_X Y
  \]
  is an isomorphism.
  Equivalently, let us show that for every $Y$-scheme $y : A \to Y$ the map of sets
  \begin{equation}\label{eq:reminiscency}
    Y^{sT}(A) \to X^{sT} \fibprod_X Y(A)
  \end{equation}
  is bijective.
  Since $p$ is étale and representable, the morphism $I_\sY \to I_\sX \fibprod_\sX \sY$ (where $\sY=[Y/T]$) is an open immersion, so in particular $\uAut_\sY(y) \to \uAut_\sX(p(y)) \fibprod_\sX \sY$ is an open immersion.
  This shows that \eqref{eq:reminiscency} is injective, so it remains to show  surjectivity.
  Let $H$ be a closed subgroup of $\uAut_\sX(x) \fibprod_\sX \sY$ which is smooth and affine over $A$ with connected $r$-dimensional fibres.
  We claim that the open immersion of group schemes over $A$
  \[ H \fibprod_{\uAut_\sX(x) \fibprod_\sX \sY} \uAut_\sY(y) \to H \]
  is invertible (and hence $H$ lifts to a closed subgroup of $\uAut_\sY(y)$ as desired).
  This can be checked over points of $A$, so we may assume that $A$ is a field.
  Now by \cite[Tag~047T]{Stacks}, this morphism is also a closed immersion, hence an inclusion of connected components.
  But $H$ is connected, so we are done.

  Since $X$ has finite tame (hence linearly reductive) stabilizers, it follows from the short exact sequence \eqref{eq:stabses} that $\sX=[X/T]$ has linearly reductive stabilizers.
  Therefore we may apply the local structure theorem of \cite[Thm.~1.11]{AlperHallHalpernLeistnerRydh} to find for every point $x$ of $X$ an affine étale neighbourhood $\sY \to \sX$ of $x$ where $\sY$ is a quotient $[V/\GL_n]$ with $V$ affine.
  By base change this gives a $T$-equivariant affine étale neighbourhood $Y \to X$.
  Since the question is étale-local on $X$ and because $X^{sT}$ is stable under the base change $Y \to X$, we may therefore replace $X$ by $Y$, so that the quotient $\sX=[X/T]$ is now a global quotient of an affine scheme by $\GL_n$.
  In this case, either $X^{sT} = \initial$ or the claim follows by combining \remref{rem:faraway} and \cite[Prop.~C.5]{EdidinRydh}.
\end{proof}

\begin{proof}[Proof of \thmref{thm:fotch}]
  Recall that $X^{rhT}$ and $X^{sT}$ are closed substacks of $X$ (see \propref{prop:cobang} and \thmref{thm:inarch}).
  Let us first show that there is an inclusion $X^{sT} \sub X^{rhT}$.
  Let $x : A \to X$ be an $A$-valued point, where $A$ is an $S$-scheme, and let $G \sub \uAut_\sX(x)$ be a closed subgroup which is smooth affine over $A$ with $r$-dimensional connected fibres.
  The composite homomorphism
  \[ G \hook \uAut_\sX(x) \xrightarrow{\alpha_A} T_A \]
  is surjective over geometric points of $A$, since $G$ has $r$-dimensional fibres and the kernel is contained in $\uAut_X(x)$ which is quasi-finite.
  It follows that the geometric fibres of $G$ are tori of rank $r$ (for every geometric point $a$ of $A$, by \cite[Cor.~1 of Prop.~11.14]{Borel} the homomorphism $G_a \to T_{a}$ restricts to a finite surjection on a maximal subtorus $H$, but then $H = G_a$ because they are smooth and connected of the same dimension).

  Since $G$ is smooth and affine with geometric fibres of constant reductive rank, it follows from \cite[Exp.~XII, Thm.~1.7(b)]{SGA3} that it admits a maximal subtorus $H \sub G$ in the sense of \cite[Exp.~XII, Def.~1.3]{SGA3}.
  But then $H = G$ since we have $H_a = G_a$ for every geometric point $a$ of $A$ and $G$ and $H$ are flat over $A$.
  In particular $G$ is a torus, which we may assume is split, since this holds étale-locally on $A$ by \cite[B.3.4]{ConradReductive} (and $X^{sT}$ and $X^{rhT}$ are subsheaves of the étale sheaf $X$).
  Now $T' := G \to T_A$ is a reparametrization.
  Using the cartesian square \eqref{eq:fnb0p1bp}, we get a group-theoretic section of the homomorphism $\uAut_{[X/T']}(x) \to T'_A$, whence the desired lift of $x : A \to X^{sT}$ to $X^{rhT}$.

  It remains to show that the inclusion $X^{sT} \sub X^{rhT}$ is an effective epimorphism.
  Take a scheme $A$ and an $A$-valued point $x : A \to X^{rhT}$ which belongs to the open and closed substack $X^{rhT}_\rho$ for some reparametrization $\rho : T' \twoheadrightarrow T$ (\propref{prop:cobang}).
  This point corresponds to a group-theoretic section $s$ of $\alpha_A : \uAut_{[X/T']}(x) \to T'_A$.
  This is a closed immersion since $\alpha_A$ is separated (as a morphism between quasi-affine schemes).
  We will show that the composite homomorphism
  \[
    s' : T'_A \xrightarrow{s} \uAut_{[X/T']}(x) \to \uAut_{[X/T]}(x)
  \]
  is a closed immersion, and hence defines an $A$-point of $X^{sT}$.
  Since the second morphism is finite, the composite $s'$ is proper, so it is enough to show that it is a monomorphism.
  This can be checked over geometric points $a$ of $A$.
  The base change $s'_a$ yields a reparametrization $T'_a \twoheadrightarrow T_a$ which by \remref{rem:noncottager} is isomorphic to the reparametrization $\uAut_{[X/T]}(x_a)^0_\red \twoheadrightarrow T_a$.
  Then $T'_a \to \uAut_{[X/T]}(x_a)^0_\red$, as a morphism between abstractly isomorphic reparametrizations of $T_a$, must itself be an isomorphism (by \remref{rem:ferial}).
  In particular, $s'_a$ is a closed immersion.
\end{proof}

\section{Perfect obstruction theories}
\label{sec:POT}

For convenience of use in applications, we describe some analogous results in the language of \cite{BehrendFantechi,AranhaPstragowski}. We first recall the notion of perfect obstruction theory in the setting of Artin stacks.
We stick to $1$-Artin stacks for simplicity, and we assume that the six functor formalism $\D$ satisfies étale descent (e.g. it is the \inftyCat of Betti sheaves, étale sheaves, or rational motives).

\begin{defn}
  Let $f: X \to Y$ be a morphism of $1$-Artin stacks in $\Stk_k$, and let $ \phi: \cE \to L_{X/Y}^{\geq -1}  \in \Dcoh(X)$. We say that $\phi$ is an {\em obstruction theory} for $f$ if $h^{i}(\phi)$ are isomorphisms for all $i \geq 0$ and $h^{-1}(\phi)$ is surjective.
  We say that $\phi$ is a {\em perfect obstruction theory} if it is an obstruction theory and $\cE \in \Dperf(X)^{\geq -1}$.
\end{defn}

We now introduce an analog of the notion of quasi-smooth in weight zero.

\begin{defn}
  Let $T$ act on $1$-Artin stacks $X,Y\in\Stk_k$ and let $f : X \to Y$ be a $T$-equivariant morphism.
  Assume that the $T$-action on $X$ is trivial.
  We say that a morphism $\phi : \cE \to L_{X/Y}^{\geq-1}$ in $\Dcoh([X/T])$ is a {\em $T$-equivariant good obstruction theory} for $f : X \to Y$ if
  \begin{defnlist}
  \item $\phi$ is an obstruction theory, and
  \item $\cE^\fix \in \Dperf^{\geq -1}(X\times BT)$ and $\cE^\mov \in \Dperf^{\geq -2}(X\times BT)$.
  \end{defnlist}
\end{defn}

The construction of $T$-equivariant Gysin pull-back in Construction~\ref{constr:gys} for quasi-smooth in weight zero morphisms can be generalized to $T$-equivariant good obstruction theories.

\begin{constr}
  Let $T$ act on $1$-Artin stacks $S\in\Stk_k$ and $X,Y\in\Stk_S$.
  Let $f : X \to Y$ be a $T$-equivariant morphism over $S$.
  Assume that $X$ is quasi-DM and the $T$-action on $X$ is trivial.
  Let $\phi:\cE \to  L_{X/Y}^{\geq-1}$ be a $T$-equivariant good obstruction theory for $f : X \to Y$.
  Then we have a closed immersion $i : \mfr{C}_{X/Y} \hookrightarrow N_{X/Y}^\vir:=\V(\cE[-1])$ by \cite[Prop.~8.2(2)]{AranhaPstragowski}.
  We define the {\em $T$-equivariant virtual pullback}
  \begin{equation*}
    f^!_T : \Chom^{\BM,T}(Y_{/S})_\loc \to \Chom^{\BM,T}(X_{/S})_\loc
  \end{equation*}
  as the composite
  \begin{multline*}
    \Chom^{\BM,T}(Y_{/S})_\loc \xrightarrow{\sp_{X/Y}} \Chom^{\BM,T}((\mfr{C}_{X/Y})_{/S})_\loc \\
    \xrightarrow{i_*} \Chom^{\BM,T}((N^\vir_{X/Y})_{/S})_\loc \simeq \Chom^{\BM,T}(X_{/S})_\loc,
  \end{multline*}
  where the specialization map $\sp_{X/Y}$ is constructed from the deformation space $M^{\circ}_{X/Y}$ in \cite[Thm.~7.2]{AranhaPstragowski} and the isomorphism is \thmref{thm:homotopy}.
\end{constr}

The $T$-equivariant virtual pull-back commutes with proper push-forwards and ordinary virtual pull-backs.

\begin{prop}\label{Prop:POT.Bivariance}
  Let $S\in\Stk_S$ be $1$-Artin and suppose given a cartesian square
  \[
    \begin{tikzcd}
    X' \ar{r}{f'}\ar{d}{p}
    & Y' \ar{d}{q}
    \\
    X \ar{r}{f}
    & Y
  \end{tikzcd} \]
  of $T$-equivariant morphisms of $1$-Artin stacks in $\Stk_S$.
  Assume that $X$ and $X'$ are quasi-DM and the $T$-actions on $X$ and $X'$ are trivial.
  Let $\phi:\cE \to  L_{X/Y}^{\geq-1}$ be a $T$-equivariant good obstruction theory for $f$.
  Then the composite $\cE|_{X'} \to  L_{X/Y}^{\geq-1}|_{X'} \to  L_{X'/Y'}^{\geq-1}$ is also a $T$-equivariant good obstruction theory for $f'$.
  \begin{thmlist}
  \item If $q$ is proper, then there is a canonical homotopy
  \[f^!_T \circ q_* \simeq p_* \circ f'^!_T: \Chom^{\BM,T}(Y'_{/S})_\loc \to \Chom^{\BM,T}(X_{/S})_\loc.\]
  \item If $q$ is equipped with a perfect obstruction theory, then there is a canonical homotopy
  \[q^! \circ f^!_T \simeq f'^!_T \circ q^!: \Chom^{\BM,T}(Y_{/S})_\loc \to \Chom^{\BM,T}(X'_{/S})_\loc.\]
  \end{thmlist}
\end{prop}

We omit the proof of \propref{Prop:POT.Bivariance} since it follows by the arguments in \ssecref{ssec:sp} by replacing the derived deformation space $D_{X/Y}$ of \cite{HekkingKhanRydh} with the classical deformation space $M^\circ_{X/Y}$ of \cite{AranhaPstragowski} (cf. \cite[Thm.~4.1]{Manolache}).

\begin{thm}\label{Thm:POT.Funct}
  Let $S\in\Stk_k$ be $1$-Artin and let $f : X \to Y$ and $g : Y \to Z$ be $T$-equivariant morphisms of $1$-Artin stacks in $\Stk_S$.
  Assume that $X$ is quasi-DM and the $T$-action on $X$ is trivial.
  Let $\phi_{X/Y}:\cE_{X/Y} \to L_{X/Y}^{\geq-1}$, $\phi_{X/Z}:\cE_{X/Z} \to L_{X/Z}^{\geq-1}$ be $T$-equivariant good obstruction theories and $\phi_{Y/Z} : \cE_{Y/Z} \to L_{Y/Z}^{\geq-1}$ be a $T$-equivariant perfect obstruction theory.
  Assume that there exists a morphism of homotopy cofiber sequences
  \[\begin{tikzcd}
  f^*\cE_{Y/Z} \ar{r}\ar{d}{f^*\phi_{Y/Z}} & \cE_{X/Z} \ar{r}\ar{d}{\phi_{X/Z}} & \cE_{X/Y} \ar{d}{\phi_{X/Y}'} \\
  (f^*L_{Y/Z}^{\geq-1})^{\geq-1} \ar{r}{a} & L_{X/Z}^{\geq-1} \ar{r} & \Cofib(a)
  \end{tikzcd}\]
  with an equivalence
  \[\phi_{X/Y} \simeq r\circ \phi_{X/Y}' : \cE_{X/Y} \to  L_{X/Y}^{\geq-1}\]
  where $r: \Cofib(a) \to \Cofib(a)^{\geq-1} \simeq L_{X/Y}^{\geq-1}$ is the canonical map.
  Then we have a canonical homotopy
  \begin{equation*}
    (g \circ f)^!_T \simeq f^!_T \circ g^! : \Chom^{\BM,T}(Z_{/S})_\loc \to \Chom^{\BM,T}(X_{/S})_\loc.
  \end{equation*}
\end{thm}

Since the arguments are almost the same as in \propref{prop:funct}, we will only give a sketch proof of \thmref{Thm:POT.Funct}.

\begin{proof}[Sketch of the proof]
  Consider the composite
  \[ h : X\times \A^1 \to Y \times \A^1 \to M^{\circ}_{Y/Z}.\]
  We claim that there is a canonical isomorphism
  \[ L_{X\times \A^1/M^\circ_{Y/Z}}^{\geq-1} \simeq \Cofib(f^*L_{Y/Z} \boxtimes\cO_{\A^1} \xrightarrow{(T,a)} (f^*L_{Y/Z} \oplus L_{X/Z})\boxtimes \cO_{\A^1})^{\geq-1}\]
  where $T \in \Gamma(\A^1,\cO_{\A^1})$ is the coordinate function.
  Indeed, when $f$ and $g$ are DM morphisms of $1$-Artin stacks, the claim is shown in \cite{KimKreschPantev}. The general case follows from descent.

  Form a morphism of exact triangles
    \[\begin{tikzcd}
    f^*\cE_{Y/Z} \ar{r}\ar{d}{f^*\phi_{Y/Z}} & f^*\cE_{Y/Z}\oplus \cE_{X/Z} \ar{r}\ar{d}{f^*\phi_{Y/Z} \oplus \phi_{X/Z}} & \cE_h \ar{d}{\phi_h'} \\
    (f^*L_{Y/Z})^{\geq-1} \ar{r}{(T,a)} & (f^*L_{Y/Z})^{\geq-1}\oplus L_{X/Z}^{\geq-1} \ar{r} & \Cofib(T,a).
    \end{tikzcd}\]
  Then the composite
  \[\phi_h : \cE_h \xrightarrow{\phi'_h} \Cofib(T,a) \to \Cofib(T,a)^{\geq-1} \simeq L_h^{\geq-1} \]
  is a $T$-equivariant good obstruction theory for $h$.

  Consider the composites
  \[k:=0_{N^\vir_{Y/Z}} \circ f : X \to N^\vir_{Y/Z}\]
  where $N^\vir_{Y/Z}:=\V(\cE_{Y/Z}[-1])$. Then $k$ has a $T$-equivariant good obstruction theory
  \[\phi_k:=f^*\phi_{Y/Z} \oplus \phi_{X/Y} : f^*\cE_{Y/Z} \oplus \cE_{X/Y} \to \tau^{\geq-1} (f^*L_{Y/Z} \oplus L_{X/Y}).\]

  By \propref{Prop:POT.Bivariance}(ii), we have a canonical homotopy
  \[(g\circ f)^!_T \simeq k^!_T \circ \sp_{Y/Z} : \Chom^{\BM,T}(Z_{/S})_\loc \to \Chom^{\BM,T}(X_{/S})_\loc.\]
  Hence it suffices to show the proposition for
  \[X \to Y \to N^\vir_{Y/Z}.\]
  By the homotopy property of $\Chom^{\BM}((-)_{/S})$, it suffices to show the proposition for
  \[X \to N^\vir_{Y/Z} \to Z.\]
  Then the analogue of \propref{prop:sppush}\itemref{item:sppull} for classical specialization maps \cite{AranhaPstragowski} and smooth pullbacks completes the proof.
\end{proof}

\begin{prop}\label{Prop:POT.selfintersection}
  Let $S\in\Stk_k$ be $1$-Artin and $f : X \to Y$ be a $T$-equivariant morphism of $1$-Artin stacks in $\Stk_S$.
  Assume that $X$ is quasi-DM and the $T$-action on $X$ is trivial.
  Let $\phi:\cE \to  L_{X/Y}^{\geq-1}$ be a $T$-equivariant good obstruction theory for $f : X \to Y$.
  If $f : X \to Y$ is a closed immersion, then we have a canonical homotopy
  \[f^!_T \circ f_* \simeq e_T(N^\vir_{\sX/\sY}) \cap (-) :\Chom^{\BM,T}(X_{/S})_\loc \to \Chom^{\BM,T}(X_{/S})_\loc\]
  where $N^\vir_{\sX/\sY}:=\V(\cE[-1])$.
\end{prop}

\propref{Prop:POT.selfintersection} follows immediately from \corref{cor:spi_*}.

\begin{cor}
  Let $X\in\Stk_k$ be a Deligne--Mumford stack with a $T$-action.
  Let $\phi:\cE \to L_{X}^{\geq-1}$ be a $T$-equivariant perfect obstruction theory.
  Choose a reparametrization $\rho : T' \to T$ such that $X^{rhT} = X^{hT'}$.
  Then the composition
  $\cE|_{X^{hT'}}^\fix \to L_X|_{X^{hT'}}^{\geq-1} \to  L_{X^{hT'}}^{\geq-1}=L_{X^{rhT}}^{\geq-1}$
  is a perfect obstruction theory for $X^{rhT}$ and is independent of the choice of $T'$.
  Moreover, we have
  \[[X]^\vir = i_*( [X^{rhT}]^\vir \cap e_T(N^\vir))^{-1} \in \Chom^{\BM,T}(X)_\loc\]
  where $i : X^{rhT} \hookrightarrow X$ denotes the inclusion map and $N^\vir :=\V(\cE|_{X^{hT'}}^\mov)$.
\end{cor}

%%%%%%%%%%%%%%%%%%%%%%%%%%%%%%%%%%%%%%%%%%%%%%%%%%%%%%%%%%%%%%%%%%%%%%%%%%%

\input{bib}

Fakultät Mathematik, Universität Duisburg-Essen, Essen, 45127, Germany

Institute of Mathematics, Academia Sinica, Taipei, 10617, Taiwan

Mathematical Institute, University of Oxford, Oxford, OX2 6GG, United Kingdom

Department of Mathematical Sciences, Seoul National University, Seoul, 08826, Korea

School of Mathematics, Tata Institute of Fundamental Research, Mumbai, 400005, India

\end{document}

%% file: preamble.tex
\usepackage[utf8]{inputenc}
\usepackage[T1]{fontenc}
\usepackage{hyperref}
\usepackage{xcolor}
\hypersetup{
    colorlinks,
    linkcolor={red!50!black},
    citecolor={blue!50!black},
    urlcolor={blue!80!black}
}

\usepackage[T1]{fontenc}
\usepackage{lmodern}

\usepackage{microtype}
\usepackage{MnSymbol}
\usepackage{textcomp}

\usepackage[mathcal]{euscript}
\let\mathcaltmp\mathcal
\usepackage{mathrsfs}
\let\mathcal\mathscr
\let\mathscr\mathcaltmp

\makeatletter
\newcommand{\eqnum}{\refstepcounter{equation}\textup{\tagform@{\theequation}}}
\makeatother

\makeatletter \@addtoreset{equation}{section} \makeatother
\renewcommand{\theequation}{\thesection.\roman{equation}}

\newtheorem{thmX}{Theorem}
\newtheorem{corX}[thmX]{Corollary}

\newtheorem{thm}{Theorem}[section]

\newtheorem*{thm*}{Theorem}
\newtheorem{lem}[thm]{Lemma}
\newtheorem{cor}[thm]{Corollary}
\newtheorem{prop}[thm]{Proposition}
\newtheorem*{defthm*}{Definition/Theorem}

\theoremstyle{definition}
\newtheorem{defn}[thm]{Definition}
\newtheorem{rem}[thm]{Remark}
\newtheorem{exam}[thm]{Example}
\newtheorem{constr}[thm]{Construction}

\newtheorem{notat}[thm]{Notation}

\newtheorem{quest}[thm]{Question}
\newtheorem*{exam*}{Example}

\newcommand\arXiv[1]{\href{http://arxiv.org/abs/#1}{arXiv:#1}}

\usepackage{xparse}

\usepackage{xspace}

\newcommand{\nc}{\newcommand}
\nc{\renc}{\renewcommand}
\nc{\ssec}{\subsection}
\nc{\sssec}{\subsubsection}
\nc{\on}{\operatorname}
\nc{\term}[1]{#1\xspace}

\usepackage{tikz}
\usetikzlibrary{matrix}
\usepackage{tikz-cd}

\tikzset{
  commutative diagrams/.cd,
  arrow style=tikz,
  diagrams={>=latex}}
\makeatletter
\tikzset{
  column sep/.code=\def\pgfmatrixcolumnsep{\pgf@matrix@xscale*(#1)},
  row sep/.code   =\def\pgfmatrixrowsep{\pgf@matrix@yscale*(#1)},
  matrix xscale/.code=%
    \pgfmathsetmacro\pgf@matrix@xscale{\pgf@matrix@xscale*(#1)},
  matrix yscale/.code=%
    \pgfmathsetmacro\pgf@matrix@yscale{\pgf@matrix@yscale*(#1)},
  matrix scale/.style={/tikz/matrix xscale={#1},/tikz/matrix yscale={#1}}}
\def\pgf@matrix@xscale{1}
\def\pgf@matrix@yscale{1}
\makeatother

\usepackage[inline]{enumitem}
\setlist[enumerate,1]{label={(\alph*)},itemsep=\parskip,leftmargin=*}

\newlist{thmlist}{enumerate}{1}
\setlist[thmlist,1]{
  label={\em(\roman*)}, ref={(\roman*)},
  itemsep=0.5em,
  align=right,widest=vi)}

\newlist{thmlistbis}{enumerate}{1}
\setlist[thmlistbis,1]{
  label={\em(\roman*~\textit{bis})},
  ref={(\roman*}~\textit{bis}\upshape{)},
  itemsep=0.5em,
  leftmargin=0pt, align=right, widest=vi)}

\newlist{defnlist}{enumerate}{2}
\setlist[defnlist,1]{
  label={(\roman*)}, ref={(\roman*)},
  itemsep=0.5em,
  topsep=0em,
  leftmargin=*,
  align=left, widest=vi)}

\setlist[defnlist,2]{
  label={(\alph*)}, ref={(\alph*)},
  itemsep=0.75em,
  labelsep=0em,labelindent=0em,leftmargin=*,align=left,widest=vi),
  topsep=0.75em}

\newlist{inlinelist}{enumerate*}{1}
\setlist[inlinelist,1]{label={(\alph*)}}

\newlist{inlinedefnlist}{enumerate*}{1}
\definecolor{green}{HTML}{38550C}
\setlist[inlinedefnlist,1]{label={\color{green}(\roman*)}}

\nc{\cA}{\ensuremath{\mathcal{A}}\xspace}
\nc{\cB}{\ensuremath{\mathcal{B}}\xspace}
\nc{\cC}{\ensuremath{\mathcal{C}}\xspace}
\nc{\cD}{\ensuremath{\mathcal{D}}\xspace}
\nc{\cE}{\ensuremath{\mathcal{E}}\xspace}
\nc{\cF}{\ensuremath{\mathcal{F}}\xspace}
\nc{\cG}{\ensuremath{\mathcal{G}}\xspace}
\nc{\cH}{\ensuremath{\mathcal{H}}\xspace}
\nc{\cI}{\ensuremath{\mathcal{I}}\xspace}
\nc{\cJ}{\ensuremath{\mathcal{J}}\xspace}
\nc{\cK}{\ensuremath{\mathcal{K}}\xspace}
\nc{\cL}{\ensuremath{\mathcal{L}}\xspace}
\nc{\cM}{\ensuremath{\mathcal{M}}\xspace}
\nc{\cN}{\ensuremath{\mathcal{N}}\xspace}
\nc{\cO}{\ensuremath{\mathcal{O}}\xspace}
\nc{\cP}{\ensuremath{\mathcal{P}}\xspace}
\nc{\cQ}{\ensuremath{\mathcal{Q}}\xspace}
\nc{\cR}{\ensuremath{\mathcal{R}}\xspace}
\nc{\cS}{\ensuremath{\mathcal{S}}\xspace}
\nc{\cT}{\ensuremath{\mathcal{T}}\xspace}
\nc{\cU}{\ensuremath{\mathcal{U}}\xspace}
\nc{\cV}{\ensuremath{\mathcal{V}}\xspace}
\nc{\cW}{\ensuremath{\mathcal{W}}\xspace}
\nc{\cX}{\ensuremath{\mathcal{X}}\xspace}
\nc{\cY}{\ensuremath{\mathcal{Y}}\xspace}
\nc{\cZ}{\ensuremath{\mathcal{Z}}\xspace}

\nc{\sA}{\ensuremath{\mathscr{A}}\xspace}
\nc{\sB}{\ensuremath{\mathscr{B}}\xspace}
\nc{\sC}{\ensuremath{\mathscr{C}}\xspace}
\nc{\sD}{\ensuremath{\mathscr{D}}\xspace}
\nc{\sE}{\ensuremath{\mathscr{E}}\xspace}
\nc{\sF}{\ensuremath{\mathscr{F}}\xspace}
\nc{\sG}{\ensuremath{\mathscr{G}}\xspace}
\nc{\sH}{\ensuremath{\mathscr{H}}\xspace}
\nc{\sI}{\ensuremath{\mathscr{I}}\xspace}
\nc{\sJ}{\ensuremath{\mathscr{J}}\xspace}
\nc{\sK}{\ensuremath{\mathscr{K}}\xspace}
\nc{\sL}{\ensuremath{\mathscr{L}}\xspace}
\nc{\sM}{\ensuremath{\mathscr{M}}\xspace}
\nc{\sN}{\ensuremath{\mathscr{N}}\xspace}
\nc{\sO}{\ensuremath{\mathscr{O}}\xspace}
\nc{\sP}{\ensuremath{\mathscr{P}}\xspace}
\nc{\sQ}{\ensuremath{\mathscr{Q}}\xspace}
\nc{\sR}{\ensuremath{\mathscr{R}}\xspace}
\nc{\sS}{\ensuremath{\mathscr{S}}\xspace}
\nc{\sT}{\ensuremath{\mathscr{T}}\xspace}
\nc{\sU}{\ensuremath{\mathscr{U}}\xspace}
\nc{\sV}{\ensuremath{\mathscr{V}}\xspace}
\nc{\sW}{\ensuremath{\mathscr{W}}\xspace}
\nc{\sX}{\ensuremath{\mathscr{X}}\xspace}
\nc{\sY}{\ensuremath{\mathscr{Y}}\xspace}
\nc{\sZ}{\ensuremath{\mathscr{Z}}\xspace}

\nc{\bA}{\ensuremath{\mathbf{A}}\xspace}
\nc{\bB}{\ensuremath{\mathbf{B}}\xspace}
\nc{\bC}{\ensuremath{\mathbf{C}}\xspace}
\nc{\bD}{\ensuremath{\mathbf{D}}\xspace}
\nc{\bE}{\ensuremath{\mathbf{E}}\xspace}
\nc{\bF}{\ensuremath{\mathbf{F}}\xspace}
\nc{\bG}{\ensuremath{\mathbf{G}}\xspace}
\nc{\bH}{\ensuremath{\mathbf{H}}\xspace}
\nc{\bI}{\ensuremath{\mathbf{I}}\xspace}
\nc{\bJ}{\ensuremath{\mathbf{J}}\xspace}
\nc{\bK}{\ensuremath{\mathbf{K}}\xspace}
\nc{\bL}{\ensuremath{\mathbf{L}}\xspace}
\nc{\bM}{\ensuremath{\mathbf{M}}\xspace}
\nc{\bN}{\ensuremath{\mathbf{N}}\xspace}
\nc{\bO}{\ensuremath{\mathbf{O}}\xspace}
\nc{\bP}{\ensuremath{\mathbf{P}}\xspace}
\nc{\bQ}{\ensuremath{\mathbf{Q}}\xspace}
\nc{\bR}{\ensuremath{\mathbf{R}}\xspace}
\nc{\bS}{\ensuremath{\mathbf{S}}\xspace}
\nc{\bT}{\ensuremath{\mathbf{T}}\xspace}
\nc{\bU}{\ensuremath{\mathbf{U}}\xspace}
\nc{\bV}{\ensuremath{\mathbf{V}}\xspace}
\nc{\bW}{\ensuremath{\mathbf{W}}\xspace}
\nc{\bX}{\ensuremath{\mathbf{X}}\xspace}
\nc{\bY}{\ensuremath{\mathbf{Y}}\xspace}
\nc{\bZ}{\ensuremath{\mathbf{Z}}\xspace}

\nc{\bbA}{\ensuremath{\mathbb{A}}\xspace}
\nc{\bbB}{\ensuremath{\mathbb{B}}\xspace}
\nc{\bbC}{\ensuremath{\mathbb{C}}\xspace}
\nc{\bbD}{\ensuremath{\mathbb{D}}\xspace}
\nc{\bbE}{\ensuremath{\mathbb{E}}\xspace}
\nc{\bbF}{\ensuremath{\mathbb{F}}\xspace}
\nc{\bbG}{\ensuremath{\mathbb{G}}\xspace}
\nc{\bbH}{\ensuremath{\mathbb{H}}\xspace}
\nc{\bbI}{\ensuremath{\mathbb{I}}\xspace}
\nc{\bbJ}{\ensuremath{\mathbb{J}}\xspace}
\nc{\bbK}{\ensuremath{\mathbb{K}}\xspace}
\nc{\bbL}{\ensuremath{\mathbb{L}}\xspace}
\nc{\bbM}{\ensuremath{\mathbb{M}}\xspace}
\nc{\bbN}{\ensuremath{\mathbb{N}}\xspace}
\nc{\bbO}{\ensuremath{\mathbb{O}}\xspace}
\nc{\bbP}{\ensuremath{\mathbb{P}}\xspace}
\nc{\bbQ}{\ensuremath{\mathbb{Q}}\xspace}
\nc{\bbR}{\ensuremath{\mathbb{R}}\xspace}
\nc{\bbS}{\ensuremath{\mathbb{S}}\xspace}
\nc{\bbT}{\ensuremath{\mathbb{T}}\xspace}
\nc{\bbU}{\ensuremath{\mathbb{U}}\xspace}
\nc{\bbV}{\ensuremath{\mathbb{V}}\xspace}
\nc{\bbW}{\ensuremath{\mathbb{W}}\xspace}
\nc{\bbX}{\ensuremath{\mathbb{X}}\xspace}
\nc{\bbY}{\ensuremath{\mathbb{Y}}\xspace}
\nc{\bbZ}{\ensuremath{\mathbb{Z}}\xspace}

\nc{\mrm}[1]{\ensuremath{\mathrm{#1}}\xspace}
\nc{\mfr}[1]{\ensuremath{\mathfrak{#1}}\xspace}
\nc{\mit}[1]{\ensuremath{\mathit{#1}}\xspace}
\nc{\mbf}[1]{\ensuremath{\mathbf{#1}}\xspace}
\nc{\mcal}[1]{\ensuremath{\mathcal{#1}}\xspace}
\nc{\msc}[1]{\ensuremath{\mathscr{#1}}\xspace}

\nc{\sub}{\subseteq}
\nc{\too}{\longrightarrow}
\nc{\hook}{\hookrightarrow}
\nc{\hooklongrightarrow}{\lhook\joinrel\longrightarrow}
\nc{\hooklong}{\hooklongrightarrow}
\nc{\hooklongleftarrow}{\longleftarrow\joinrel\rhook}
\nc{\twoheadlongrightarrow}{\relbar\joinrel\twoheadrightarrow}
\nc{\longrightleftarrows}{\ \raisebox{0.3ex}{\(\mathrel{\substack{\xrightarrow{\rule{1em}{0em}} \\[-1ex] \xleftarrow{\rule{1em}{0em}}}}\)}\ }

\renc{\ge}{\geqslant}
\renc{\le}{\leqslant}

\nc{\id}{\mathrm{id}}
\DeclareMathOperator{\Ker}{\on{Ker}}

\DeclareMathOperator{\rk}{\mathrm{rk}}
\DeclareMathOperator{\Hom}{\on{Hom}}
\nc{\uHom}{\underline{\smash{\Hom}}}
\DeclareMathOperator{\Maps}{\on{Maps}}
\DeclareMathOperator{\Aut}{\on{Aut}}
\DeclareMathOperator{\End}{\on{End}}

\nc{\uEnd}{\underline{\smash{\End}}}

\nc{\colim}{\varinjlim}
\renc{\lim}{\varprojlim}
\nc{\Cofib}{\on{Cofib}}
\nc{\Fib}{\on{Fib}}
\nc{\initial}{\varnothing}
\nc{\op}{\mathrm{op}}

\DeclareMathOperator*{\fibprod}{\times}

\renc{\setminus}{\smallsetminus}

\usepackage{mathtools}
\DeclarePairedDelimiter\abs{\lvert}{\rvert}%

\newcommand{\thmref}[1]{Theorem~\ref{#1}}

\newcommand{\secref}[1]{Sect.~\ref{#1}}
\newcommand{\ssecref}[1]{Subsect. ~\ref{#1}}

\newcommand{\lemref}[1]{Lemma~\ref{#1}}
\newcommand{\propref}[1]{Proposition~\ref{#1}}
\newcommand{\corref}[1]{Corollary~\ref{#1}}

\newcommand{\remref}[1]{Remark~\ref{#1}}
\newcommand{\defnref}[1]{Definition~\ref{#1}}

\renewcommand{\eqref}[1]{(\ref{#1})}
\newcommand{\constrref}[1]{Construction~\ref{#1}}

\newcommand{\examref}[1]{Example~\ref{#1}}

\newcommand{\notatref}[1]{Notation~\ref{#1}}
\newcommand{\itemref}[1]{\ref{#1}}

%% file: macros.tex
\nc{\A}{\bA}
\renc{\P}{\bP}
\nc{\Spec}{\on{Spec}}
\nc{\D}{\on{\mbf{D}}}
\nc{\SH}{\on{\mbf{SH}}}
\nc{\DM}{\on{\mbf{DM}}}
\nc{\MGLmod}{\on{\mbf{D}_\MGL}}
\nc{\MGLmodQ}{\on{\mbf{D}_{\MGL,\Q}}}
\nc{\cat}{\mrm{1}}
\nc{\AHR}{\mrm{AHR}}
\nc{\geom}{\mrm{T}}
\nc{\Dqc}{\on{\mbf{D}}_{\mrm{qc}}}
\nc{\bDelta}{\mathbf{\Delta}}
\nc{\Cech}{\textnormal{\v{C}}}
\nc{\Dperf}{\on{\mbf{D}}_{\mrm{perf}}}
\nc{\Coh}{\on{Coh}}
\nc{\Qcoh}{\on{Qcoh}}
\nc{\Dcoh}{\on{\mbf{D}}_{\mrm{coh}}}
\nc{\uCoh}{\underline{\smash{\Coh}}}
\nc{\Higgs}{\underline{\smash{\on{Higgs}}}}
\nc{\cl}{{\mrm{cl}}}
\nc{\Bl}{\on{Bl}}
\nc{\vir}{\mrm{vir}}
\nc{\CH}{\on{CH}}
\nc{\HC}{\on{A}} % H of C (our notation for the (co)homology of the (co)chain complexes)
\nc{\et}{\mrm{\acute{e}t}}
\renc{\H}{\on{H}}
\nc{\BM}{\mrm{BM}}
\nc{\Z}{\bZ}
\nc{\Q}{\bQ}
\nc{\K}{{\on{K}}}
\nc{\KGL}{\mrm{KGL}}
\nc{\MGL}{\mrm{MGL}}
\nc{\KB}{\K^{\mrm{B}}}
\nc{\G}{{\on{G}}}
\nc{\KH}{\mrm{KH}}
\nc{\Ket}{\K^{\et}}
\nc{\KHet}{\KH^{\et}}
\nc{\Get}{\G^{\et}}
\nc{\rightrightrightarrows}
  {\mathrel{\substack{\textstyle\rightarrow\\[-0.6ex]
   \textstyle\rightarrow \\[-0.6ex]
   \textstyle\rightarrow}}}
\nc{\rightrightrightrightarrows}
  {\mathrel{\substack{\textstyle\rightarrow\\[-0.6ex]
   \textstyle\rightarrow \\[-0.6ex]
   \textstyle\rightarrow \\[-0.6ex]
   \textstyle\rightarrow}}}
\nc{\Einfty}{{\sE_\infty}}
\renc{\sp}{\mrm{sp}}
\nc{\Td}{\on{Td}}
\nc{\ch}{\on{ch}}
\nc{\RGamma}{R\Gamma}
\nc{\red}{\mrm{red}}
\nc{\der}{{\mrm{der}}}
\nc{\Mod}{{\mrm{Mod}}}
\nc{\Gr}{{\on{Gr}}}
\nc{\Ind}{\on{Ind}}
\nc{\form}{\widehat}
\nc{\R}{\bR}
\renc{\L}{\bL}
\nc{\otimesL}{\mathchoice{\overset{\bL}{\otimes}}{\otimes^\bL}{\otimes^\bL}{\otimes^\bL}}
\nc{\fibprodR}{\fibprod^R}
\nc{\uRHom}{\bR\uHom}
\nc{\GL}{\mrm{GL}}
\nc{\SL}{\mrm{SL}}
\nc{\SW}{\on{SW}}
\nc{\Vect}{\on{Vect}}
\nc{\Fun}{\on{Fun}}
\nc{\vb}[1]{\langle #1\rangle}
\nc{\loc}{\mrm{loc}}
\nc{\fix}{\mrm{fix}}
\nc{\mov}{\mrm{mov}}
\nc{\cms}{\mrm{cms}}
\nc{\V}{\bV}
\nc{\Gm}{{\bG_m}}
\nc{\pt}{\mrm{pt}}
\nc{\mot}{\mrm{mot}}
\nc{\an}{\mrm{an}}
\nc{\St}{\mrm{St}}
\nc{\pr}{\mrm{pr}}
\nc{\C}{\on{C}}
\nc{\Chom}{\mrm{C}_\bullet}
\nc{\Chomhat}{\widehat{\mrm{C}}_\bullet}
\nc{\Ccoh}{\mrm{C}^\bullet}
\nc{\Ccohhat}{\widehat{\mrm{C}}^\bullet}
\nc{\CBM}{\mrm{C}^{\BM}_\bullet}
\nc{\uAut}{\underline{\Aut}}
\nc{\per}{\vb{\ast}} % "periodization"
\nc{\Lis}{\mrm{Lis}}
\nc{\aff}{{\mrm{aff}}}
\nc{\dR}{{\mrm{dR}}}
\nc{\Pic}{{\on{Pic}}}
\nc{\uGrp}{\underline{\smash{\mrm{Grp}}}}
\nc{\uPerf}{\underline{\smash{\mrm{Perf}}}}
\nc{\uPic}{\underline{\smash{\Pic}}}
\nc{\uMap}{\underline{\smash{\mrm{Map}}}}
\nc{\uDiv}{\underline{\smash{\mrm{Div}}}}
\nc{\uPair}{\underline{\smash{\mrm{Pair}}}}
\nc{\dash}{\textnormal{-}}
\nc{\nilp}{\mrm{nilp}}
\nc{\Rep}{\on{R}}
\nc{\dashmod}{\dash\mbf{mod}}
\nc{\bLambda}{\mbf{\Lambda}}
\renc{\max}{\mrm{max}}
\nc{\un}{\mbf{1}}
\nc{\Stk}{\mrm{Stk}}
\nc{\dStk}{\mrm{dStk}}
\nc{\qDM}{\mrm{qDM}}

\nc{\scr}{\term{derived commutative ring}}
\nc{\scrs}{\term{derived commutative rings}}

\nc{\inftyCat}{\term{$\infty$-category}}
\nc{\inftyCats}{\term{$\infty$-categories}}

\nc{\inftyGrpd}{\term{$\infty$-groupoid}}
\nc{\inftyGrpds}{\term{$\infty$-groupoids}}

\nc{\dA}{\term{derived Artin}}

%% file: bib.tex
%!TEX root = virloc.tex

\bibliographystyle{halphanum}

%% file: virloc.bbl
\begin{thebibliography}{EHKSY}

  \bibitem[AHHR]{AlperHallHalpernLeistnerRydh} J.~Alper, J.~Hall, D.~Halpern-Leistner, D.~Rydh, \textit{Artin algebraization for pairs with applications to the local structure of stacks and Ferrand pushouts}.  Forum~Math.~Sigma~{\bf{12}} (2024), no.~e20.

  \bibitem[AHR]{AlperHallRydhLuna} J.~Alper, J.~Hall, D.~Rydh, \textit{A Luna étale slice theorem for algebraic stacks}, Ann.~Math.~(2)~{\bf{191}}, no. 3, 675--738 (2020).

  \bibitem[AKLPR]{constack} D.~Aranha, A.\,A.~Khan, A.~Latyntsev, H.~Park, C.~Ravi, \textit{The stacky concentration theorem}.  \arXiv{2407.08747} (2024).

  \bibitem[AP]{AranhaPstragowski} D.~Aranha, P.~Pstrągowski, \textit{The Intrinsic Normal Cone For Artin Stacks}.  Ann. Inst. Fourier~{\bf{74}} (2024), no.~1, 71--120.

  \bibitem[Alp]{AlperGood} J.~Alper, \textit{Good moduli spaces for Artin stacks}.  Ann. Inst. Fourier~{\bf{63}} (2013), no.~6, 2349--2402.

  \bibitem[BF]{BehrendFantechi} K.~Behrend, B.~Fantechi, \textit{The intrinsic normal cone}, Invent.~Math.~{\bf{128}} (1997), no. 1, 45--88.

  \bibitem[Bor]{Borel} A.~Borel, \textit{Linear algebraic groups}.  Grad. Texts Math.~{\bf{126}} (1991).

  \bibitem[CD]{CisinskiDegliseBook} D.-C.~Cisinski, F.~Déglise, \textit{Triangulated categories of mixed motives}. Springer Monogr. Math.~(2019), 406 p.

  \bibitem[CKL]{ChangKiemLi} H.~Chang, Y.~Kiem, J.~Li, \textit{Torus localization and wall crossing for cosection localized virtual cycles}.  Adv. Math.~{\bf{308}} (2017), 964--986.
  
  \bibitem[Con]{ConradReductive} B.~Conrad, \textit{Reductive group schemes}.  Available at: \url{http://math.stanford.edu/~conrad/papers/luminysga3.pdf} (2014).

  \bibitem[EG1]{EdidinGraham} D.~Edidin, W.~Graham, \textit{Equivariant intersection theory}.  Invent.~Math.~{\bf{131}} (1998), no. 3, 595--644.

  \bibitem[EG2]{EdidinGrahamLocalization} D.~Edidin, W.~Graham, \textit{Localization in equivariant intersection theory and the Bott residue formula}.  Am.~J.~Math.~{\bf{120}} (1998), no.~3, 619--636.

  \bibitem[EGA]{EGA} A.~Grothendieck and J.~Dieudonn\'e,
    {\it \'El\'ements de g\'eom\'etrie alg\'ebrique},
    Publ. Math. IH\'ES
    {\bf 4} (Chapter 0, 1--7, and I, 1--10),
    {\bf 8} (II, 1--8),
    {\bf 11} (Chapter 0, 8--13, and III, 1--5),
    {\bf 17} (III, 6--7),
    {\bf 20} (Chapter 0, 14--23, and IV, 1), {\bf 24} (IV, 2--7), {\bf 28} (IV, 8--15),  and  {\bf 32} (IV, 16--21), 1960--1967.

  \bibitem[EK]{ElmantoKhan} E.~Elmanto, A.\,A.~Khan, \textit{Perfection in motivic homotopy theory}. Proc. Lond. Math. Soc.~{\bf{120}} (2020), no.~1, 28--38.

  \bibitem[ER]{EdidinRydh} D.~Edidin, D.~Rydh, \textit{Canonical reduction of stabilizers for Artin stacks with good moduli spaces}.  Duke Math. J.~{\bf{170}} (2021), no.~5, 827--880.

  \bibitem[FYZ]{FengYunZhang} T.~Feng, Z.~Yun, W.~Zhang, \textit{Higher theta series for unitary groups over function fields}.  \arXiv{2110.07001} (2021).

  \bibitem[Fog]{Fogarty} J.~Fogarty, \textit{Fixed point schemes}, Amer. J. Math.~{\bf{95}} (1973), 35--51.

  \bibitem[GP]{GraberPandharipande} T.~Graber, R.~Pandharipande, \textit{Localization of virtual classes}. Invent.~Math.~{\bf{135}} (1999), no.~2, 487--518.

  \bibitem[HKR]{HekkingKhanRydh} J.~Hekking, A.\,A.~Khan, D.~Rydh, \textit{Deformation to the normal cone and blowups via derived Weil restrictions}.  In preparation.
  
  \bibitem[HML]{HellerMalagonLopez} J.~Heller, J.~Malag\'on-L\'opez, \textit{Equivariant algebraic cobordism}.  J. Reine Angew. Math.~{\bf{684}}, 87--112 (2013).

  \bibitem[HLP]{HalpernLeistnerPreygel} D.~Halpern-Leistner, A.~Preygel, \textit{Mapping stacks and categorical notions of properness}.  Compositio Math.~{\bf{159}} (2023), 530--589.
    
  \bibitem[Ive]{Iversen} B. Iversen, \textit{A fixed point formula for actions of tori on algebraic varieties}.  Invent. Math.~{\bf{16}} (1972), 229-236.

  \bibitem[Joy]{JoyceEnumerative} D.~Joyce, \textit{Enumerative invariants and wall-crossing formulae in abelian categories}.  \arXiv{2111.04694} (2021).

  \bibitem[KKP]{KimKreschPantev} B.~Kim, A.~Kresch, T.~Pantev, \textit{Functoriality in intersection theory and a conjecture of Cox, Katz, and Lee}. J. Pure Appl. Algebra~{\bf{179}} (2003), no.~1-2, 127--136.

  \bibitem[KL]{KiemLiWall} Y.-H.~Kiem, J.~Li, \textit{A wall crossing formula of Donaldson-Thomas invariants without Chern-Simons functional}.  Asian J. Math.~{\bf{17}} (2013), no.~1, 63--94.

  \bibitem[Kha1]{KhanVirtual} A.\,A.~Khan, \textit{Virtual fundamental classes for derived stacks I}. \arXiv{1909.01332} (2019).

  \bibitem[Kha2]{KhanKstack} A.\,A.~Khan, \textit{K-theory and G-theory of derived algebraic stacks}.  Jpn. J. Math.~{\bf{17}} (2022), 1--61.
  
  \bibitem[Kon]{Kontsevich} M.~Kontsevich, \textit{Enumeration of rational curves via torus actions}, in: The moduli space of curves (Texel Island, 1994), 335--368, Progr. Math.~{\bf{129}} (1995).

  \bibitem[Kre]{Kresch} A.~Kresch, \textit{Cycle groups for Artin stacks}.  Invent. Math.~{\bf{138}} (1999), no.~3, 495--536.

  \bibitem[Lur1]{LurieHA} J.~Lurie, \textit{Higher algebra}, version of 2017-09-18. \url{https://www.math.ias.edu/~lurie/papers/HA.pdf}.

  \bibitem[Lur2]{LurieSAG} J.~Lurie, \textit{Spectral algebraic geometry}, version of 2018-02-03. \url{https://www.math.ias.edu/~lurie/papers/SAG-rootfile.pdf}.

  \bibitem[Mad]{Madapusi} K. Madapusi, \textit{Derived special cycles on Shimura varieties}.  \arXiv{2212.12849} (2022).

  \bibitem[Man]{Manolache} C.~Manolache, \textit{Virtual pull-backs}. J. Algebraic Geom.~{\bf{21}} (2012), no. 2, 201--245.

  \bibitem[Mou]{Moulinos} T.~Moulinos, \textit{The geometry of filtrations}.  Bull. Lond. Math. Soc.~{\bf{53}} (2021), no.~5, 1486--1499.

  \bibitem[PY1]{PortaYuAnalytic} M.~Porta, T.\,Y.~Yu, \textit{Derived non-Archimedean analytic spaces}.  Sel. Math.~{\bf{24}} (2018), no.~2, 609--665.

  \bibitem[PY2]{PortaYuGW} M.~Porta, T.\,Y.~Yu, \textit{Non-archimedean Gromov--Witten invariants}.  \arXiv{2209.13176} (2022).

  \bibitem[Ray]{RaynaudFaisceaux} M.~Raynaud, \textit{Faisceaux amples sur les schemas en groupes et les espaces homogenes}.  Lect. Notes Math.~{\bf{119}} (1970).

  \bibitem[Rom1]{EffectiveModels} M.~Romagny, \textit{Effective models for group schemes} J. Algebraic Geom.~{\bf{21}} (2012), no 4, 643--682.
  
  \bibitem[Rom2]{RomagnyGroupActions} M.~Romagny, \textit{Group Actions on Stacks and Applications}.  Michigan Math. J.~{\bf{53}} (2005).

  \bibitem[Rom3]{RomagnyFixedMult} M.~Romagny, \textit{Fixed point stacks under groups of multiplicative type}.  \arXiv{2101.02450} (2021).

  \bibitem[Rom4]{RomagnyApp} M.~Romagny, \textit{Fixed-point-reflecting morphisms, after Drinfeld and Alper-Hall-Rydh}.  Appendix A in: A.~Mayeux, \textit{Algebraic Magnetism}.  \arXiv{2203.08231} (2022).

  \bibitem[Rom5]{FixedLocusRomagny} M.~Romagny, \textit{Algebraicity and smoothness of fixed point stacks}.  \arXiv{2205.11114} (2022).

  \bibitem[Ryd]{RydhSemicontinuous} D.~Rydh, \textit{Answer to MathOverflow question MO/193}.  Available at: \url{https://mathoverflow.net/a/202} (version of 2013-03-13).

  \bibitem[SGA3]{SGA3} M.~Demazure, A.~Grothendieck (eds.), \textit{Schémas en groupes} (SGA 3). Séminaire de Géometrie Algébrique du {B}ois-{M}arie 1962--1964. Lecture Notes in Mathematics~{\bf{151}}, {\bf{152}}, {\bf{153}}, Springer (1970).
  
  \bibitem[SP]{Stacks} The Stacks Project. \url{https://stacks.math.columbia.edu}.

  \bibitem[Tho]{ThomasonHomotopy} R.\,W.~Thomason, \textit{The homotopy limit problem}, in: Contemp.~Math.~{\bf{19}} (1983), 407--419.

\end{thebibliography}
